\documentclass[11pt]{article}

\usepackage[T1]{fontenc}
\usepackage{lmodern}
\usepackage[utf8]{inputenc}
\usepackage{microtype}
\usepackage{framed}
\usepackage{listings}
\usepackage{vmargin}
\usepackage{setspace}
\usepackage{mathrsfs, mathenv}
\usepackage{amsmath, amsthm, amssymb, amsfonts, amscd}
\usepackage{graphicx}
\usepackage{epstopdf}
\usepackage[svgnames]{xcolor}
\usepackage{hyperref}
\hypersetup{citecolor=blue, colorlinks=true, linkcolor=black}
\usepackage[ruled, vlined]{algorithm2e}
\usepackage[capitalise]{cleveref}
\usepackage{subcaption}
\usepackage{bbm}
\usepackage{cite}
\usepackage{verbatim}
\usepackage{pgfplots}
\usepackage{tikz}
\usetikzlibrary{arrows,decorations.pathmorphing,backgrounds,positioning,fit,matrix}
\usepackage{etoolbox}
\usepackage{color}
\usepackage{lipsum}
\usepackage{ifthen}
\usepackage[title]{appendix}
\usepackage{autonum}
\usepackage{accents}
\usepackage{xpatch}
\usepackage[]{algpseudocode}
\usepackage{multicol}
\usepackage[abs]{overpic}
\usepackage{eqnarray}

\crefname{assumption}{Assumption}{Assumptions}
\crefname{figure}{Figure}{Figures}

\theoremstyle{plain}
\newtheorem{theorem}{Theorem}[section]

\newtheorem{lemma}[theorem]{Lemma}

\numberwithin{equation}{section}

\theoremstyle{definition}

\theoremstyle{remark}
\newtheorem{remark}[theorem]{Remark}
\newtheorem{assumption}[theorem]{Assumption}
\newtheorem{example}[theorem]{Example}

\ifpdf
  \DeclareGraphicsExtensions{.eps,.pdf,.png,.jpg}
\else
  \DeclareGraphicsExtensions{.eps}
\fi

\usepackage{mathtools}

\usepackage{booktabs}

\usepackage{pgfplots}
\usepackage{tikz}
\usetikzlibrary{patterns,arrows,decorations.pathmorphing,backgrounds,positioning,fit,matrix}
\usepackage[labelfont=bf]{caption}
\usepackage{here}
\usepackage[font=normal]{subcaption}

\usepackage{enumitem}
\setlist[itemize]{leftmargin=.5in}
\setlist[enumerate]{leftmargin=.5in,topsep=3pt,itemsep=3pt,label=(\roman*)}

\newcommand{\email}[1]{\href{#1}{#1}}
\newcommand{\TheTitle}{A method of moments estimator for interacting particle systems and their mean field limit}
\newcommand{\TheAuthors}{G. A. Pavliotis, A. Zanoni}
\title{\TheTitle}
\author{Grigorios A. Pavliotis \thanks{Department of Mathematics, Imperial College London, London SW7 2AZ, UK, \email{g.pavliotis@imperial.ac.uk}.}
\and Andrea Zanoni \thanks{Institute of Mathematics, École Polytechnique Fédérale de Lausanne, 1015 Lausanne, Switzerland, \email{andrea.zanoni@epfl.ch}. Institute for Computational and Mathematical Engineering, Stanford University, Stanford, CA, USA, \email{andrea.zanoni@stanford.edu}.}}
\date{}

\usepackage{amsopn}

\newcommand{\abs}[1]{\left\lvert#1\right\rvert}
\newcommand{\norm}[1]{\left\|#1\right\|}
\renewcommand{\Pr}{\mathbb{P}}

\newcommand{\N}{\mathbb{N}}

\newcommand{\R}{\mathbb{R}}

\newcommand{\defeq}{\coloneqq}
\newcommand{\eqdef}{\eqqcolon}

\newcommand{\E}{\operatorname{\mathbb{E}}}

\newcommand{\argmin}{\operatorname{argmin}}

\newcommand{\compl}{\mathsf{C}}

\renewcommand{\d}{\mathrm{d}}
\newcommand{\dd}{\,\mathrm{d}}
\definecolor{shade}{RGB}{100, 100, 100}
\definecolor{bordeaux}{RGB}{128, 0, 50}

\usepackage[usestackEOL]{stackengine}

\definecolor{leg1}{RGB}{0,114,189}
\definecolor{leg2}{RGB}{217,83,25}
\definecolor{leg3}{RGB}{237,177,32}
\definecolor{leg4}{RGB}{126,47,142}
\definecolor{leg5}{RGB}{119,172,48}

\definecolor{leg21}{RGB}{62,38,169}
\definecolor{leg22}{RGB}{46,135,247}
\definecolor{leg23}{RGB}{55,200,151}
\definecolor{leg24}{RGB}{254,195,56}

\ifpdf
\hypersetup{
	pdftitle={\TheTitle},
	pdfauthor={\TheAuthors}
}
\fi

\begin{document}
	
\maketitle	

\begin{abstract} 
We study the problem of learning unknown parameters in stochastic interacting particle systems with polynomial drift, interaction and diffusion functions from the path of one single particle in the system. Our estimator is obtained by solving a linear system which is constructed by imposing appropriate conditions on the moments of the invariant distribution of the mean field limit and on the quadratic variation of the process. Our approach is easy to implement as it only requires the approximation of the moments via the ergodic theorem and the solution of a low-dimensional linear system. Moreover, we prove that our estimator is asymptotically unbiased in the limits of infinite data and infinite number of particles (mean field limit). In addition, we present several numerical experiments that validate the theoretical analysis and show the effectiveness of our methodology to accurately infer parameters in systems of interacting particles.
\end{abstract}

\textbf{AMS subject classifications.} 35Q70, 35Q83, 35Q84, 60J60, 62M20, 65C30.

\textbf{Key words.} Interacting particle system, mean field limit, inference, Fokker–Planck equation, moments, ergodicity, linear system, polynomials.

\section{Introduction}

The problem of inference for stochastic differential equations (SDEs) has a long history~\cite{BaR80,Rao99,Kut04,Sor04,Bis08,Iac08}. Numerous different techniques, such as maximum likelihood estimation and Bayesian approaches, have been developed and analyzed extensively, both in a parametric and non-parametric setting. Most of the work in the aforementioned references is focused on inference for linear, in the sense of McKean, SDEs, i.e., that the drift and diffusion coefficients do not depend on the law of the process. In recent years, several studies have been devoted to the problem of learning parameters in nonlinear SDEs, for which the drift and possibly also the diffusion coefficients depend on the law of the process~\cite{LaL23,DeH23,MeB22,AHP23,CoG22,YCY22,GeL22}. The interest in inference for McKean SDEs, and of the corresponding McKean--Vlasov partial differential equations (PDEs), is due to the many applications in which such nonlinear models appear, for example systemic risk~\cite{GPY13}, collective behaviour such as flocking and swarming~\cite{NPT10}, pedestrian dynamics~\cite{GSW19}, opinion formation~\cite{GGS22}, and neuroscience~\cite{Dai19}, to name but a few. McKean nonlinear SDEs appear in the mean field limit of systems of weakly interacting diffusions~\cite{Oel84,Daw83}. The link between finite dimensional diffusion processes, describing the interacting particle system, and their mean field limit has been used in order to study the maximum likelihood estimator~\cite{Kas90,Bis11}.

In earlier work, we studied the problem of inference for McKean SDEs using two different approaches. In particular, in~\cite{SKP23} we used the stochastic gradient descent method, both in an online and offline setting, and in~\cite{PaZ22} we employed the eigenfunction martingale estimator~\cite{KeS99} to learn parameters in both the confining and interaction potentials, as well as the diffusion coefficient, of the McKean SDE in one dimension. More precisely, we showed that, under the assumptions from~\cite{Mal01}, the eigenfunction martingale estimator, obtained by solving the eigenvalue problem for the generator of the McKean SDE linearized around the (unique) stationary state, is asymptotically unbiased and normal in the limit as the number of particles and observations go to infinity. We emphasize the fact that, for our methodology to work, it is sufficient to observe only one particle. Furthermore, even though the theoretical analysis is valid under the assumption that the mean field dynamics has a unique stationary state, we demonstrated by means of numerical experiments that our method works also when multiple stationary states exist, i.e., when the mean field dynamics exhibit phase transitions~\cite{CGP20}.

We notice that the methodology proposed and analyzed in~\cite{PaZ22}, even though it is elegant, requires observation of only one particle and is provably asymptotically unbiased and normal, it suffers from the drawback that it is computationally expensive. This is due to the fact that we need to solve repeatedly the eigenvalue problem for the generator of the linearized McKean SDE. This means, in particular, that it might be impractical to apply our method to problems where many parameters need to be learned from data, as well as in the multidimensional setting. In this paper, we propose a very simple, robust and computationally efficient methodology that can be applied to McKean SDEs with polynomial nonlinearities. In particular, we show that the classical method of moments can be employed to learn parameters in McKean SDEs with polynomial drift and diffusion coefficients. Similarly to the eigenfunction martingale estimator, for our method to work, it is sufficient to observe a single particle of the interacting particle system, whose mean field limit leads to the McKean SDE. For our theoretical analysis, we need to consider continuous, as opposed to discrete, observations, at least when the diffusion coefficient is not known. Furthermore, we need to assume that the McKean SDE has a unique invariant measure, even though numerical experiments suggest that this is not necessary in practice. Under these assumptions, we can show that our estimator is asymptotically unbiased when the number of particles and the time horizon tend to infinity, and provide a rate of convergence.

The method of moments is a very standard methodology in statistics~\cite[Section 3.3]{Pan16}. It has also been used in the study of the (linear and nonlinear) Fokker--Planck equation, as a numerical technique~\cite{WiB70}, as well as a mode reduction method~\cite{ZPL23}. Moreover, it was also used in~\cite{Daw83} for proving rigorously that the Desai--Zwanzig model exhibits a phase transition. To our knowledge, this simple and straightforward approach has not yet been applied to the problem of inference for McKean SDEs. The empirical moments have been employed in \cite{BPP23} to construct a kernel based estimator to approximate a polynomial interaction function in McKean SDEs. However, their approach is different from ours and they assume to observe all the particles only at the final time instead of the full path of a single particle. Moreover, they do not include a confining potential in their analysis, and they assume the diffusion to be constant, while we consider more general polynomial confining potentials and diffusion functions. We also mention the recent work~\cite{GeL22}, in which the inference problem for McKean SDEs in one dimension with no confining potential, but with a polynomial interaction potential, was studied. Under the assumptions of stationarity and for continuous-time observations, an approximate likelihood function was constructed using the empirical moments of the invariant distribution. Even though this work is related to ours, our approach and the result obtained in this paper are different. First, we consider the case with a confining potential, and we estimate parameters in both the confining and the interaction potentials, as well as in the diffusion coefficient, then we use directly the empirical moments for the estimation of the parameters in the drift and diffusion coefficients, without having to construct an approximate likelihood. Moreover, we apply our methodology to problems where the invariant measure of the mean field dynamics is not unique (i.e., when the system exhibits phase transitions), and we also show that our approach applies also to problems with a non-gradient structure and with degenerate noise. Finally, when the diffusion coefficient is known, we demonstrate that it is sufficient to consider discrete observations, in order to estimate parameters in the drift.

We finally remark that our setting where the confining and interaction potentials and the diffusion function have a polynomial form is not a strong limitation of the scope of this work. In fact, e.g., phase transitions for long time behavior of interacting particles with linear interaction and quadratic diffusion has been studied in \cite{GrH07}. Moreover, there are several McKean SDEs that appear in applications and that have polynomial drift coefficients, such as the Desai--Zwanzig model~\cite{Daw83} and the interacting Fitzhugh--Nagumo neurons model~\cite{Dai19}. In the numerical experiments we demonstrate the usefulness of the method of moments by applying it to both models.

\subsection*{Our main contributions}

We now summarize the main contributions of this work.

\begin{itemize}[leftmargin=0.5cm]
\item We present a novel approach for inferring parameters in interacting particle systems when the drift, interaction and diffusion functions are polynomial. We assume to observe the path of one single particle and we construct a linear system for the parameters to be estimated, and which depends on the moments of the invariant distribution of the mean field limit and on the quadratic variation of the process.
\item We prove that our estimator is asymptotically unbiased in the sense that it converges in $L^1$ to the true value of the unknown parameters when the observation time and the number of interacting particles go to infinity. The proof is based on an a priori forward error stability analysis for the linear system for the unknown parameters.
\item We present various numerical experiments with the multiple purpose of validating the theoretical analysis, showing the accuracy of our estimator and demonstrating numerically that our approach can be employed even in cases where some of the assumptions are not satisfied.
\end{itemize}

\subsection*{Outline}

The rest of the paper is organized as follows. In \cref{sec:problem} we introduce the model under investigation, the inference problem and the main assumptions. In \cref{sec:method} we present the method of moments and in \cref{sec:analysis} we prove theoretically the asymptotic unbiasedness of our estimator. In \cref{sec:experiments} we show several numerical experiments to illustrate the advantages of our methodology, and finally in \cref{sec:conclusion} we draw our conclusions and outline possible developments of this approach.

\section{Problem setting} \label{sec:problem}

In this work we consider the following system of one-dimensional interacting particles over the time interval $[0,T]$
\begin{equation} \label{eq:systemIP}
\begin{aligned}
\d X_t^{(n)} &= f(X_t^{(n)};\alpha) \dd t + \frac1N \sum_{i=1}^{N} g(X_t^{(n)} - X_t^{(i)};\gamma) \dd t + \sqrt{2 h(X_t^{(n)};\sigma)} \dd B_t^{(n)}, \\
X_0^{(n)} &\sim \nu, \qquad n = 1, \dots, N,
\end{aligned}
\end{equation}
where $N$ is the number of particles, $\{ (B_t^{(n)})_{t\in[0,T]} \}_{n=1}^N$ are standard independent one dimensional Brownian motions, and $\nu$ is the initial distribution of the particles, which is assumed to be independent of the Brownian motions $\{ (B_t^{(n)})_{t\in[0,T]} \}_{n=1}^N$. We remark that we assume chaotic initial conditions, meaning that all the particles are initially distributed according to the same measure. The functions $f,g$ and $h$ are the drift, interaction, and diffusion functions, respectively, which depend on some parameters $\alpha \in \R^{J+1}, \gamma \in \R^{K+1}$ and $\sigma \in \R^{L+1}$. In this article we consider polynomial functions of the form
\begin{equation}
f(x;\alpha) = \sum_{j=0}^J \alpha_j x^j, \qquad g(x;\gamma) = \sum_{k=0}^K \gamma_k x^k, \qquad \text{and} \qquad h(x;\sigma) = \sum_{\ell=0}^L \sigma_\ell x^\ell,
\end{equation}
and we denote by $\theta \in \Theta \subseteq \R^P$ the subvector of $\upsilon = \begin{pmatrix} \alpha^\top & \gamma^\top & \sigma^\top \end{pmatrix}^\top \in \R^{J+K+L+3}$ with the unknown components which we aim to estimate. Moreover, $\Theta$ is the set of admissible parameters and $\theta^*$ is the exact value of the vector $\theta$. Our goal is then to infer $\theta^*$ from the continuous path of the realization of one single particle $(X_t^{(\bar n)})_{t\in[0,T]}$ of system \eqref{eq:systemIP}.
\begin{remark}
Considering polynomial functions is not a strong limitation of the scope of our work. In fact, this setting can be seen as a semiparameteric framework where we aim to estimate the whole drift, interaction, and diffusion functions, provided that they are sufficiently smooth and can therefore be approximated by polynomials.
\end{remark}
We focus our attention on large systems, i.e., when the number of interacting particles is $N \gg 1$, and therefore it is reasonable to look at the mean field limit (see, e.g., \cite{Daw83,Gar88}), which provides a good approximation of the behavior of a single particle in the system. In particular, letting $N \to \infty$ in \eqref{eq:systemIP} we obtain the nonlinear, in the sense of McKean, SDE
\begin{equation} \label{eq:MFL}
\begin{aligned} 
\d X_t &= f(X_t;\alpha) \dd t + (g(\cdot;\gamma) * u(\cdot,t))(X_t) \dd t + \sqrt{2h(X_t;\sigma)} \dd B_t, \\
X_0 &\sim \nu,
\end{aligned}
\end{equation}
where $u(\cdot;t)$ is the density with respect to the Lebesgue measure of the law of $X_t$, and satisfies the nonlinear Fokker--Planck equation named McKean--Vlasov equation
\begin{equation} 
\begin{aligned}
\frac{\partial u}{\partial t}(x,t) &= -\frac{\partial}{\partial x} \left( f(x;\alpha)u(x,t) + (g(\cdot;\gamma)*u(\cdot,t))(x)u(x,t) \right) + \frac{\partial^2}{\partial x^2} \left( h(x;\sigma) u(x,t) \right), \\
u(x,0) \dd x &= \nu(\d x).
\end{aligned}
\end{equation}
We notice that the SDE \eqref{eq:MFL} is said to be nonlinear in the sense that the drift function depends on the law of the process. Moreover, the mean field can have multiple invariant measures $\mu(\d x) = \rho(x) \dd x$ (see, e.g., \cite{Daw83,CGP20}), whose densities with respect to the Lebesgue measure $\rho$ solve the stationary Fokker--Planck equation
\begin{equation} \label{eq:FPstationary}
- \frac{\d}{\d x} \left( f(x;\alpha)\rho(x) + (g(\cdot;\gamma)*\rho)(x)\rho(x) \right) + \frac{\d^2}{\d x^2} \left( h(x;\sigma) \rho(x) \right)  = 0.
\end{equation}
However, for the following convergence analysis we will work under the assumption that equation \eqref{eq:FPstationary} admits a unique solution. Nevertheless, we observe from the numerical experiments in \cref{sec:num_bistable} that this assumption does not restrict the range of applicability of the method proposed in this paper from a practical point of view, indeed we notice that technique presented here works even in presence of multiple invariant states. The main hypotheses needed in the analysis are then summarized below.
\begin{assumption} \label{as:analysis}
The set of admissible parameters $\Theta \subseteq \R^P$ is convex and compact. Moreover, for all $\theta \in \Theta$ it holds:
\begin{itemize}[leftmargin=0.5cm]
\item $h(x;\sigma) > 0$ for all $x \in \R$,
\item the solution $X_t$ of \eqref{eq:MFL} is ergodic with unique invariant measure $\mu(\d x) = \rho(x) \dd x$ whose density $\rho$ solves equation \eqref{eq:FPstationary}.
\end{itemize}
\end{assumption}
\begin{example} \label{ex:Malrieu}
\cref{as:analysis} is satisfied if we consider, e.g., the setting of \cite{Mal01}. In particular, the uniqueness of the invariant measure of the mean field limit is given in the case of additive noise, i.e., $h(x;\sigma) = \sigma_0 > 0$, and when the drift and interaction functions have the form $f(x;\alpha) = -V'(x;\alpha)$ and $g(x;\gamma) = -W'(x;\gamma)$ for some confining potential $V(\cdot;\alpha) \in \mathcal C^2(\R)$ which is uniformly convex (i.e., there exists $\beta>0$ such that $V''(x) \ge \beta$ for all $x \in \R$) and some interaction potential $W(\cdot;\gamma) \in \mathcal C^2(\R)$ which is even and convex. In this case the density $\rho$ of the invariant measure is the solution of
\begin{equation}
\rho(x) = \frac1Z \mathrm{exp} \left\{ - \frac1{\sigma_0} (V(x;\alpha) + (W(\cdot;\kappa)*\rho)(x)) \right\},
\end{equation}
where $Z$ is the normalization constant
\begin{equation}
Z = \int_\R \mathrm{exp} \left\{ - \frac1{\sigma_0} (V(x;\alpha) + (W(\cdot;\kappa)*\rho)(x)) \right\} \dd x.
\end{equation}
Moreover, in \cite[Theorem 3.18]{Mal01}, it is proved that there exists a constant $C>0$ independent of time such that
\begin{equation}
\norm{u(\cdot,t) - \rho}_{L^1(\R)} \le Ce^{-\beta t/2},
\end{equation}
and therefore the density $u(\cdot,t)$ converges exponentially fast to the invariant density $\rho$.
\end{example}
\begin{remark}
For the clarity of the exposition we focus here on systems of $N$ interacting particles in one dimension. Nevertheless, the proposed method can be easily extended to the case of $d$-dimensional particles with $d>1$ as we will see in the numerical experiment in \cref{sec:num_DaiPra}.
\end{remark}

\section{Method of moments} \label{sec:method}

We now present our approach for learning the unknown parameter $\theta \in \Theta \subseteq \R^P$ from continuous observations of one single particle of the system \eqref{eq:systemIP}. We assume to know a realization $(X_t^{(\bar n)})_{t\in[0,T]}$ of the $\bar n$-th particle of the solution of \eqref{eq:systemIP} for some $\bar n = 1, \dots, N$, and we aim to retrieve the exact value of the unknown parameter $\theta$ by means of the method of moments. Letting $M \in \N$, we first multiply equation \eqref{eq:FPstationary} by $x^m$ for $m = 1, \dots, M$, integrate over $\R$ and then by parts to obtain 
\begin{equation} \label{eq:moments_m}
\int_\R x^{m-1} f(x;\alpha) \rho(x) \dd x + \int_\R x^{m-1} (g(\cdot;\gamma)*\rho)(x) \rho(x) \dd x + (m-1) \int_\R x^{m-2} h(x;\sigma) \rho(x) \dd x = 0.
\end{equation}
Replacing the definitions of $f$ and $g$ yields for the drift term
\begin{equation} \label{eq:drift_moment}
\int_\R x^{m-1} f(x;\alpha) \rho(x) \dd x = \sum_{j=0}^J \alpha_j \mathbb M^{(m+j-1)},
\end{equation}
and for the diffusion term
\begin{equation} \label{eq:diffusion_moment}
\int_\R x^{m-2} h(x;\sigma) \rho(x) \dd x = \sum_{\ell=0}^L \sigma_\ell \mathbb M^{(m+\ell-2)},
\end{equation}
where $\mathbb M^{(r)}$ denotes the $r$-th moment with respect to the invariant measure $\mu$
\begin{equation}
\mathbb M^{(r)} = \E^{\mu}[X^r].
\end{equation}
Moreover, due to the binomial theorem, for the interaction term we have
\begin{equation} \label{eq:interaction_moment}
\begin{aligned}
\int_\R x^{m-1} (g(\cdot;\gamma)*\rho)(x) \rho(x) \dd x &= \sum_{k=0}^K \gamma_k \int_\R \int_\R x^{m-1} (x-y)^k \rho(y) \rho(x) \dd y \dd x \\
&= \sum_{k=0}^K \gamma_k \sum_{i=0}^k \binom{k}{i} \mathbb M^{(m+i-1)} \mathbb M^{(k-i)}.
\end{aligned}
\end{equation}
Collecting equation \eqref{eq:moments_m} together with the expansions \eqref{eq:drift_moment}, \eqref{eq:diffusion_moment} and \eqref{eq:interaction_moment} we obtain
\begin{equation} \label{eq:equation_moments}
\sum_{j=0}^J \alpha_j \mathbb M^{(m+j-1)} + \sum_{k=0}^K \gamma_k \sum_{i=0}^k (-1)^{k-i} \binom{k}{i} \mathbb M^{(m+i-1)} \mathbb M^{(k-i)} + \sum_{\ell=0}^L \sigma_\ell (m-1) \mathbb M^{(m+\ell-2)} = 0,
\end{equation}
and recalling the definition of $\theta$ as the subvector of $\upsilon = \begin{pmatrix} \alpha^\top & \gamma^\top & \sigma^\top \end{pmatrix}^\top \in \R^{J+K+L+3}$ with the unknown components, we can write a linear system for the exact value of $\theta$ of the form 
\begin{equation} \label{eq:system_exact_tilde}
\mathscr M \theta^* = \mathfrak b, 
\end{equation}
where the matrix $\mathscr M \in \R^{M \times P}$ and the right-hand side $\mathfrak b \in \R^{M}$ depend on the remaining known components of $\upsilon$ and on the moments $\mathbb M^{(r)}$ with $r = 0, \dots, M + \max \{ J-1, K-1, L-2 \}$. Hence, we obtained a set of constraints which have to be satisfied by the exact unknown. However, since we do not know the invariant measure $\mu$ and therefore its moments $\mathbb M^{(r)}$, we cannot construct the matrix $\mathscr M$ and the right-hand side $\mathfrak b$ in order to solve the linear system and get the value of $\theta$. Nevertheless, due to the ergodicity of the process, the exact moments can be approximated using the data by
\begin{equation}
\widetilde{\mathbb M}^{(r)}_{T,N} = \frac1T \int_0^T (X_t^{(\bar n)})^r \dd t,
\end{equation}
which leads to a first attempt for the definition of our estimator as the solution $\bar \theta^M_{T,N}$ of the approximated linear system
\begin{equation} \label{eq:linear_system_tilde}
\widetilde {\mathscr M}_{T,N} \bar \theta^M_{T,N} = \widetilde{\mathfrak b}_{T,N},
\end{equation}
where the matrix $\widetilde {\mathscr M}_{T,N}$ and the right-hand side $\widetilde{\mathfrak b}_{T,N}$ are obtained from $\mathscr M$ and $\mathfrak b$, respectively, by replacing the exact moments $\mathbb M^{(r)}$ with their approximation $\widetilde{\mathbb M}^{(r)}_{T,N}$. Let us now remark that it may happen for the right-hand side to be $\widetilde{\mathfrak b}_{T,N} = 0$, if, e.g., all the parameters are unknown ($\theta = \upsilon$). In this case it is necessary to exclude the trivial solution $\bar \theta^M_{T,N} = 0$, because it is clearly not the one which we are looking for. Hence, we can then augment the system with an additional equation derived from the definition of the quadratic variation. In particular, let $Q_{T,N}$ be the quadratic variation of the process $(X_t^{(\bar n)})_{t \in [0,T]}$ and notice that
\begin{equation}
Q_{T,N} = 2 \int_0^T h(X_t^{(\bar n)};\sigma) \dd t = 2 \sum_{\ell=0}^L \sigma_\ell \int_0^T (X_t^{(\bar n)})^\ell \dd t,
\end{equation}
which implies 
\begin{equation} \label{eq:equationQV}
\sum_{\ell=0}^L \sigma_\ell \widetilde{\mathbb M}_{T,N}^{(\ell)} = \frac{Q_{T,N}}{2T}.
\end{equation}
Letting $\Delta>0$ be sufficiently small, then the quadratic variation $Q_{T,N}$ can be approximated by the quantity 
\begin{equation} \label{eq:approximationQV}
\widetilde Q_{T,N} = \sum_{i=1}^I (X_{(i+1)\Delta}^{(\bar n)} - X_{i\Delta}^{(\bar n)})^2,
\end{equation}
where $I = T/\Delta$. Since we know a continuous trajectory $(X_t^{(\bar n)})_{t \in[0,T]}$, then $\Delta$ can be chosen arbitrarily small, and hence we assume $Q_{T,N}$ to be known exactly. We are now ready to define our estimator starting from the solution $\widetilde \theta^M_{T,N} \in \Theta$ of the linear system
\begin{equation} \label{eq:linear_system_hat}
\widetilde{\mathcal M}_{T,N} \widetilde \theta^M_{T,N} = \widetilde b_{T,N},
\end{equation}
which is obtained by adding equation \eqref{eq:equationQV} to the system \eqref{eq:linear_system_tilde} and where therefore $\widetilde{\mathcal M}_{T,N} \in \R^{(M+1)\times P}$ and $\widetilde b_{T,N} \in \R^{M+1}$. We remark that the value $M$ has to be chosen such that the number of equations of the system \eqref{eq:linear_system_hat} is greater or equal than the number of parameters, i.e., $M \ge P-1$. Then, if $M = P-1$ and the matrix $\widetilde{\mathcal M}_{T,N}$ has full rank, the system has a unique solution, otherwise we can compute the minimum-norm least squares solution through the Moore–Penrose pseudoinverse. Moreover, since the parameter $\theta^* \in \Theta$, the estimator $\widehat \theta^M_{T,N} \in \Theta$ is finally defined as the projection of the solution of the linear system onto the convex and compact space $\Theta$
\begin{equation}
\widehat \theta^M_{T,N} = \mathcal P_\Theta(\widetilde \theta_{T,N}^M) = \mathcal P_\Theta (\widetilde{\mathcal M}_{T,N}^\dagger \widetilde b_{T,N}) \eqdef \argmin_{\theta \in \Theta} \norm{\theta - \widetilde{\mathcal M}_{T,N}^\dagger \widetilde b_{T,N}},
\end{equation}
where, for a matrix $A$, we denote by $A^\dagger$ its Moore–Penrose pseudoinverse. We remark that this definition is necessary in order to make the estimator well-defined but, as we will see later in the numerical experiments, it is not required in concrete applications in most of the cases. In fact, the least squares solution of system \eqref{eq:linear_system_hat} turns out to be always unique in practice. Moreover, for $N$ and $T$ large enough, it is also sufficiently close to the exact parameter, and therefore inside the set of admissible values $\Theta$, so that the projection is not needed. For completeness and since it will be useful in the following analysis, let us also introduce the linear system whose unique solution is the exact unknown parameter $\theta^*$
\begin{equation} \label{eq:system_exact_hat}
\mathcal M \theta^* = b,
\end{equation}
which is obtained by adding to the system \eqref{eq:system_exact_tilde} the following limit equation for the quadratic variation $Q_T$ of the mean field limit
\begin{equation} \label{eq:equationQV_limit}
\sum_{\ell=0}^L \sigma_\ell \mathbb M^{(\ell)} = Q \defeq \lim_{T\to\infty} \frac{Q_T}{2T}.
\end{equation}
In \cref{alg:procedure} we summarize the main steps needed to construct the estimator $\widehat \theta^M_{T,N}$.

\begin{algorithm}
\caption{Estimation of $\theta \in \Theta \subset \R^P$} \label{alg:procedure}
\begin{tabbing}
\hspace*{\algorithmicindent} \textbf{Input:} \= Drift function $f(\cdot;\alpha)$ with $\alpha \in R^{J+1}$. \\
\> Interaction function $g(\cdot;\gamma)$ with $\gamma \in R^{K+1}$. \\
\> Diffusion function $h(\cdot;\sigma)$ with $\sigma \in R^{L+1}$. \\
\> Observed particle $(X_t^{(\bar n)})_{t\in[0,T]}$. \\
\> Number of moments equations $M$.
\end{tabbing}
\begin{tabbing}
\hspace*{\algorithmicindent} \textbf{Output:} \= Estimation $\widehat \theta_{T,N}^M$ of $\theta$.
\end{tabbing}
\begin{enumerate}[label=\arabic*:,itemsep=5pt]
\item Compute the quadratic variation $Q_{T,N}$ using equation \eqref{eq:approximationQV} with \\ $\Delta$ sufficiently small.
\item Compute the approximated moments $\widetilde{\mathbb M}_{T,N}^{(r)}$ for $r = 0, \dots, M + \max \{ J-1, K-1, L-2 \}$.
\item Construct the matrix $\widetilde{\mathcal M}_{T,N}$ and the right-hand side $\widetilde b_{T,N}$ using \\ equations \eqref{eq:equation_moments} and \eqref{eq:equationQV} and steps 1 and 2.
\item Let $\widehat \theta_{T,N}^M$ be the projection onto $\Theta$ of the minimum-norm least squares solution of the linear system $\widetilde{\mathcal M}_{T,N} \theta = \widetilde b_{T,N}$.
\end{enumerate}
\end{algorithm}

\begin{remark}
The method of moments is outlined here in the context of continuous-time observations. Nevertheless, we remark that, as long as the diffusion term is known, the proposed methodology can be easily generalized to the case of discrete-time observations for which a numerical experiment is presented in \cref{sec:comparison_OU}. In fact, the approximation based on the ergodic theorem of the exact moments with respect to the invariant measure of the mean field dynamics can be repeated analogously with discrete-time observations instead of continuous-time observations.
\end{remark}

\begin{remark}
In this work we decided to consider the case where the path of one particle is observed, rather than of all the particles, since in concrete applications it is more likely to get measurements only from one single particle than from all the ensemble of particles. Nevertheless, we stress that if the path of more than one particle is observed, then we may follow two different approaches in order to use all the available information and improve the performances of our estimator. First, we could compute the method of moments estimator for each particle, and then take the average as final estimator. Otherwise, we could also estimate the moments by averaging the empirical moments computed for each particle, and finally solve the linear system. We believe that this second approach could potentially give better results, but the analysis of the method of moments estimator for multiple observable particles is out of the scope of this work.
\end{remark}

\subsection{The mean field Ornstein--Uhlenbeck process} \label{sec:OU}

We now employ the method of moments to a simple example involving the mean field Ornstein--Uhlenbeck process for which the exact moments can be computed analytically. We set $J = 1$, $K = 1$, $L = 0$ with $\alpha_0 = 0$, $\gamma_0 = 0$, $\gamma_1 = - 1$, so that we have
\begin{equation}
f(x;\alpha) = \alpha_1 x, \qquad g(x;\gamma) = -x, \qquad h(x;\sigma) = \sigma_0,
\end{equation}
where $\alpha_1 < 0$, $\sigma_0>0$, and therefore we obtain the interacting particle system
\begin{equation} \label{eq:system_OU}
\begin{aligned}
\d X_t^{(n)} &= \alpha_1 X_t^{(n)} \dd t - \left( X_t^{(n)} - \frac1N \sum_{i=1}^{N} X_t^{(i)} \right) \dd t + \sqrt{2 \sigma_0} \dd B_t^{(n)}, \\
X_0^{(n)} &\sim \nu, \qquad n = 1, \dots, N,
\end{aligned}
\end{equation}
whose mean field limit reads
\begin{equation} \label{eq:MFL_OU}
\begin{aligned}
\d X_t &= \alpha_1 X_t \dd t - \left( X_t - \E[X_t] \right) \dd t + \sqrt{2 \sigma_0} \dd B_t, \\
X_0 &\sim \nu.
\end{aligned}
\end{equation}
We then aim to estimate the two-dimensional parameter $\theta = \begin{pmatrix} \alpha_1 & \sigma_0 \end{pmatrix}^\top$. We remark that the inference problem for this particular case has been investigated in many works such as \cite{Kas90,Bis11}. It can be shown (see, e.g., \cite{Daw83,Fra05,GoP18}) that the process $X_t$ has a unique invariant measure 
\begin{equation}
\mu = \mathcal N \left( 0, \frac{\sigma_0}{1 - \alpha_1} \right),
\end{equation}
and therefore the moments $\mathbb M^{(r)}$ for $r \in \N$ are given by
\begin{equation} \label{eq:momentsOU}
\mathbb M^{(r)} = \begin{cases}
0 & \text{if $r$ is odd}, \\
\left( \dfrac{\sigma_0}{1 - \alpha_1} \right)^{r/2} (r-1)!! & \text{if $r$ is even}.
\end{cases}
\end{equation}
Moreover, equations \eqref{eq:equation_moments} and \eqref{eq:equationQV_limit} simplify to
\begin{equation}
\begin{aligned}
\alpha_1 \mathbb M^{(m)} + \sigma_0(m-1)\mathbb M^{(m-2)} &= \mathbb M^{(m)}, \qquad m=1,\dots,M, \\
\sigma_0 &= Q,
\end{aligned}
\end{equation}
which together with \eqref{eq:momentsOU} and the fact that $Q = \sigma_0$ allows us to write explicitly the linear system \eqref{eq:system_exact_hat}. In particular, we notice that only the even terms give useful equations and, taking $m = 2, 4, \dots, M$ with $M$ even, we get
\begin{equation} \label{eq:exact_b_OU}
b = \begin{pmatrix} \frac{\sigma_0}{1 - \alpha_1} & 3 \left( \frac{\sigma_0}{1-\alpha_1} \right)^2 & \cdots & \left( \frac{\sigma_0}{1-\alpha_1} \right)^{M/2} (M-1)!! & \sigma_0 \end{pmatrix}^\top \in \R^{M/2+1},
\end{equation}
and
\begin{equation} \label{eq:exact_M_OU}
\mathcal M = \begin{pmatrix}
\frac{\sigma_0}{1-\alpha_1} & 3 \left( \frac{\sigma_0}{1-\alpha_1} \right)^2 & \cdots & \left( \frac{\sigma_0}{1-\alpha_1} \right)^{M/2} (M-1)!! & 0 \\
1 & 3\left( \frac{\sigma_0}{1-\alpha_1} \right) & \cdots & \left( \frac{\sigma_0}{1-\alpha_1} \right)^{M/2-1} (M-1)!! & 1
\end{pmatrix}^\top \in \R^{(M/2+1) \times 2}.
\end{equation}
We can easily verify that the system is solved by $\theta$ and that the least squares solution is well posed as $\det(\mathcal M^\top \mathcal M) \neq 0$. Indeed, the matrix $\mathcal M$ can be rewritten as 
\begin{equation}
\mathcal M = \begin{pmatrix}
v & \frac{1 - \alpha_1}{\sigma_0} v \\
0 & 1
\end{pmatrix}, \qquad \text{where} \qquad v = \begin{pmatrix}
\frac{\sigma_0}{1-\alpha_1} & 3 \left( \frac{\sigma_0}{1-\alpha_1} \right)^2 & \cdots & \left( \frac{\sigma_0}{1-\alpha_1} \right)^{M/2} (M-1)!!
\end{pmatrix}^\top,
\end{equation}
which gives
\begin{equation}
\mathcal A = \mathcal M^\top \mathcal M = \begin{pmatrix}
\norm{v}^2 & \frac{1 - \alpha_1}{\sigma_0} \norm{v}^2 \\
\frac{1 - \alpha_1}{\sigma_0} \norm{v}^2 & \left(\frac{1 - \alpha_1}{\sigma_0} \right)^2 \norm{v}^2 + 1,
\end{pmatrix},
\end{equation}
which in turn implies $\det(\mathcal A) = \det(\mathcal M^\top \mathcal M) = \norm{v}^2 \neq 0$ since $v \neq 0$. 
 Finally, the other systems \eqref{eq:system_exact_tilde}, \eqref{eq:linear_system_tilde} and \eqref{eq:linear_system_hat} are then obtained in a similar way.

\begin{remark} \label{rem:OU_identifiable}
We assumed the interaction coefficient $\gamma_1$ to be known in order to guarantee the solvability of the problem. Otherwise it would only be possible to estimate the sum $\alpha_1 + \gamma_1$. In fact, the method of moments relies on the mean field limit \eqref{eq:MFL_OU} at stationarity, i.e., when $\E[X_t] = 0$, which would depend only on the sum $\alpha_1 + \gamma_1$ and not on the single parameters alone. Indeed, if we had assumed both $\alpha_1$ and $\gamma_1$ to be unknown, then we would have got $\mu = \mathcal N \left( 0, -\sigma_0/(\alpha_1 + \gamma_1) \right)$ as invariant measure, and the matrix $\mathcal A$ would have been
\begin{equation}
\mathcal A = \begin{pmatrix}
\norm{v}^2 & \norm{v}^2 & - \frac{\alpha_1 + \gamma_1}{\sigma_0} \norm{v}^2 \\
\norm{v}^2 & \norm{v}^2 & - \frac{\alpha_1 + \gamma_1}{\sigma_0} \norm{v}^2 \\
- \frac{\alpha_1 + \gamma_1}{\sigma_0} \norm{v}^2 & - \frac{\alpha_1 + \gamma_1}{\sigma_0} \norm{v}^2 & \left( - \frac{\alpha_1 + \gamma_1}{\sigma_0} \right)^2 \norm{v}^2 + 1
\end{pmatrix}.
\end{equation}
Therefore, we would have obtained $\det(\mathcal A) = 0$, in contrast with the assumption, independently of the number $M$ of moments equations.
\end{remark}

\section{Convergence analysis} \label{sec:analysis}

In this section we study the asymptotic properties of the proposed estimator. In particular, we prove that it is asymptotically unbiased when both the final time and the number of particles tend to infinity, and we provide a rate of convergence. We first need to introduce an additional assumption, which guarantees that the uniform propagation of chaos property holds true for the system of interacting particles and that the variables $X_t^{(n)}$ and $X_t$ have bounded moments of any order for all $t \in [0,T]$ and $n = 1, \dots, N$.
\begin{assumption} \label{as:propagation_chaos}
There exists a constant $C>0$ independent of $T$ and $N$ such that for all $t \in [0,T]$ and $n=1, \dots,N$
\begin{equation}
\left( \E \left[ (X_t^{(n)} - X_t)^2 \right] \right)^{1/2} \le \frac{C}{\sqrt{N}},
\end{equation}
where $(X_t^{(n)})_{t\in[0,T]}$ is one particle of the system \eqref{eq:systemIP} and $(X_t)_{t\in[0,T]}$ is the solution of the corresponding mean field limit \eqref{eq:MFL}, where the Brownian motion $(B_t)_{t\in[0,T]}$ is chosen to be $(B_t^{(n)})_{t\in[0,T]}$.
Moreover, it holds for all $q \ge 1$
\begin{equation}
\left( \E \left[ \abs{X_t^{(n)}}^q \right] \right)^{1/q} \le C \qquad \text{and} \qquad \left( \E \left[ \abs{X_t}^q \right] \right)^{1/q} \le C.
\end{equation}
\end{assumption}
Even if a general theory for McKean SDEs with polynomial coefficients which guarantees \cref{as:analysis,as:propagation_chaos} does not exist, these hypotheses are satisfied in several frameworks in the literature. First, we can consider the setting of \cref{ex:Malrieu}, and consequently the mean field Ornstein--Uhlenbeck process in \cref{sec:OU}, where propagation of chaos and bounded moments are shown in \cite[Theorem 3.3]{Mal01} and \cite[Lemma 2.3.1]{Gan08}. In \cite{Daw83} the special case, where the confining potential is bistable, the interaction potential is quadratic, and the diffusion is constant, is studied in full detail and \cref{as:analysis,as:propagation_chaos} are proven together with the analysis of the phase transition. Still maintaining the diffusion coefficient constant, another interesting class of models, which has been studied in several papers \cite{Mal03,CGM08,LLM21,LMT22,BPP23}, is obtained by setting the drift term equal to zero and considering an odd nondecreasing interaction term. This setting has been thoroughly analyzed in \cite{BRT98,BRV98}. The authors show the existence and uniqueness of the solution of the McKean SDE (\cite[Theorem 3.1]{BRT98}), and the existence and uniqueness of a stationary distribution for the mean field limit (\cite[Theorems 4.1 and 4.7]{BRT98}) with some particular cases explicitly considered in \cite[Propositions 4.13 and 4.14]{BRT98}. Moreover, they study the interacting particle system and the corresponding propagation of chaos result (\cite[Proposition 5.1 and Theorem 5.3]{BRT98}), and finally prove the convergence to the stationary distribution (\cite[Theorems 3.1 and 4.1]{BRV98}). Fewer papers consider the case of nonconstant diffusion term. We mention the work in \cite{Lac18} where the propagation of chaos result is proved under the assumptions that the drift is Lipschitz with respect to the measure, the diffusion coefficient is nondegenerate and the ratio between the drift and the diffusion is bounded. Moreover, a more general setting, where also the diffusion coefficient can potentially depend on the law of the process, has been considered recently in \cite{GHL23}. Here, propagation of chaos is shown assuming different nonglobally Lipschitz conditions for both the drift and the diffusion terms. We would be interested in further investigating existence and uniqueness of solution and invariant measure for McKean SDEs and the corresponding interacting particle system for more general drift and diffusion terms. We believe that this would be possible at least for processes where the drift and the interaction coefficients are gradients of convex functions and the diffusion is bounded below by a positive constant. We think that the starting point for an idea on how to prove these properties consists in restricting the whole space to a ball of radius $R$, so that the theory in \cite[Section 5]{Mal01}, which is developed for bounded and globally Lipschitz functions, can be applied. Then, the extension to the whole space could be obtained by letting $R\to\infty$ and combining a strong dissipativity condition for the confining potential and the theory of Lyapunov functions in \cite{MSH02} together with the localization lemma in \cite[Chapter 3]{PrR07}. However, these last results must first be extended to the case of mean field SDEs. We will return to this problem in future work.
\begin{remark}
We notice that whenever we write $\norm{A}$ and $\norm{A}_{\mathrm{F}}$ for a real matrix $A$ we mean its spectral and Frobenius norm, respectively. Moreover, all the constants will be denoted by $C$ even if their value can change from line to line.
\end{remark}
The next lemma is a technical result which shows the convergence of the approximated moments $\widetilde{\mathbb M}^{(r)}_{T,N}$ towards the exact moments $\mathbb M^{(r)}$ of the invariant measure $\mu$ of the mean field limit, and will be crucial for the proof of the main theorem.
\begin{lemma} \label{lem:convergence_moments}
Let \cref{as:analysis,as:propagation_chaos} hold and let $\nu = \mu$. Then, for all $r \in \N$ and for all $q \in [1,2)$ there exists a constant $C>0$ independent of $T$ and $N$ such that
\begin{equation}
\left( \E \left[ \abs{\widetilde{\mathbb M}_{T,N}^{(r)} - \mathbb M^{(r)}}^q \right] \right)^{1/q} \le C \left( \frac{1}{\sqrt{T}} + \frac{1}{\sqrt{N}} \right).
\end{equation}
\end{lemma}
\begin{proof}
By the Minkowsky inequality we first have
\begin{equation} \label{eq:decompositionI}
\begin{aligned}
\left( \E \left[ \abs{\widetilde{\mathbb M}_{T,N}^{(r)} - \mathbb M^{(r)}}^q \right] \right)^{1/q} &= \left( \E \left[ \abs{\frac1T \int_0^T (X_t^{(n)})^r \dd t - \E^{\mu}[X^r]}^q \right] \right)^{1/q} \\
&\le \left( \E \left[ \abs{\frac1T \int_0^T (X_t^{(n)})^r \dd t - \frac1T \int_0^T (X_t)^r \dd t}^q \right] \right)^{1/q} \\
&\quad + \left( \E \left[ \abs{\frac1T \int_0^T (X_t)^r \dd t - \E^{\mu}[X^r]}^q \right] \right)^{1/q} \\
&\eqdef I_1 + I_2,
\end{aligned}
\end{equation}
and then we analyze the two terms in the right-hand side separately. Since the initial distribution $\nu$ is equal to the invariant distribution $\mu$ of the mean field limit, then equation \eqref{eq:MFL} becomes a standard Itô SDE and we can apply the Hölder inequality and the ergodic theorem in \cite[Section 4]{MST10} to obtain
\begin{equation} \label{eq:boundI2}
I_2 \le \left( \E \left[ \abs{\frac1T \int_0^T (X_t)^r \dd t - \E^{\mu}[X^r]}^2 \right] \right)^{1/2} \le \frac{C}{\sqrt{T}}.
\end{equation}
We then consider $I_1$ and applying Hölder's inequality we have
\begin{equation}
I_1^q = \frac{1}{T^q} \E \left[ \abs{\int_0^T \left( (X_t^{(n)})^r - (X_t)^r \right) \dd t}^q \right] \le \frac1T \int_0^T \E \left[ \abs{(X_t^{(n)})^r - (X_t)^r}^q \right] \dd t,
\end{equation}
which implies
\begin{equation} \label{eq:boundI1q}
I_1^q \le \sup_{t\in[0,T]} \E \left[ \abs{(X_t^{(n)})^r - (X_t)^r}^q \right] \eqdef \sup_{t\in[0,T]} E_t.
\end{equation}
Applying Hölder's and Jensen's inequality we get
\begin{equation}
\begin{aligned}
E_t &= \E \left[ \abs{X_t^{(n)} - X_t}^q \abs{(X_t^{(n)})^{r-1} + (X_t^{(n)})^{r-2} X_t + \dots + X_t^{(n)}(X_t)^{r-2} + (X_t)^{r-1}}^q \right] \\
&\le r^{\frac{3q-2}{2}} \left( \E \left[ \abs{X_t^{(n)} - X_t}^2 \right] \right)^{\frac{q}{2}} \\
&\hspace{1cm} \times \left\{ \E \left[ \abs{X_t^{(n)}}^{\frac{2q(r-1)}{2-q}} \right] + \left( \E \left[ \abs{X_t^{(n)}}^{\frac{4q(r-2)}{2-q}} \right] \right)^{1/2} \left( \E \left[ \abs{X_t}^{\frac{4q}{2-q}} \right] \right)^{1/2} + \dots \right. \\
&\hspace{2cm} + \left. \left( \E \left[ \abs{X_t^{(n)}}^{\frac{4q}{2-q}} \right] \right)^{1/2} \left( \E \left[ \abs{X_t}^{\frac{4q(r-2)}{2-q}} \right] \right)^{1/2} + \E \left[ \abs{X_t}^{\frac{2q(r-1)}{2-q}} \right] \right\}^{\frac{2-q}{2}},
\end{aligned}
\end{equation}
and due to the boundedness of the moments of $X_t^{(n)}$ and $X_t$ and the uniform propagation of chaos property in \cref{as:propagation_chaos} we obtain
\begin{equation}
E_t \le C \left( \E \left[ \abs{X_t^{(n)} - X_t}^2 \right] \right)^{q/2} \le C N^{-q/2},
\end{equation}
which together with bound \eqref{eq:boundI1q} yields
\begin{equation} \label{eq:boundI1}
I_1 \le \frac{C}{\sqrt N}.
\end{equation}
Finally, decomposition \eqref{eq:decompositionI} and bounds \eqref{eq:boundI2} and \eqref{eq:boundI1} give the desired result.
\end{proof}
\begin{remark}
The technical assumption in \cref{lem:convergence_moments} that the particles in the system are initially distributed according to the invariant measure of the corresponding mean field limit is not a serious limitation of the validity of the lemma, as long as the system is ergodic and satisfies the uniform propagation of chaos property, which is equivalent to having exponentially fast convergence to equilibrium, both for the $N$-particle system and for the mean field SDE \cite{DGP23,GLW22}. This can also be observed in the numerical experiments presented in the next section, where the choice of the initial condition does not affect the performance of the estimator.
\end{remark}
Let us now focus on the system for the exact unknown \eqref{eq:system_exact_hat} and the system \eqref{eq:linear_system_hat} obtained using the available data, and consider the corresponding least squares linear systems
\begin{equation}
\mathcal A \theta^* = c \qquad \text{and} \qquad \widetilde{\mathcal A}_{T,N} \widetilde \theta_{T,N}^M = \widetilde c_{T,N},
\end{equation}
where $\mathcal A, \widetilde{\mathcal A}_{T,N} \in \R^{P \times P}$ and $c, \widetilde c_{T,N} \in \R^P$ are defined as
\begin{equation}
\mathcal A = \mathcal M^\top \mathcal M, \qquad c = \mathcal M^\top b, \qquad \widetilde{\mathcal A}_{T,N} = \widetilde{\mathcal M}_{T,N}^\top \widetilde{\mathcal M}_{T,N}, \qquad \widetilde c_{T,N} = \widetilde{\mathcal M}_{T,N}^\top \widetilde b_{T,N}.
\end{equation}
We remark that if the matrix $\widetilde{\mathcal A}_{T,N}$ has full rank, then the linear system has a unique solution which coincides with the minimum-norm least squares solution of the system \eqref{eq:linear_system_hat}. In fact, since it holds
\begin{equation}
\widetilde{\mathcal M}_{T,N}^\dagger = (\widetilde{\mathcal M}_{T,N}^\top \widetilde{\mathcal M}_{T,N})^{-1} \widetilde{\mathcal M}_{T,N}^\top,
\end{equation}
then we have
\begin{equation}
\widetilde \theta_{T,N}^M = \widetilde{\mathcal A}_{T,N}^{-1} \widetilde c_{T,N} = \widetilde{\mathcal M}_{T,N}^\dagger \widetilde b_{T,N}.
\end{equation}
Notice also that if we write
\begin{equation} \label{eq:system_perturbation}
\left[ \mathcal A + \left( \widetilde{\mathcal A}_{T,N} - \mathcal A \right)  \right] \left[ \theta^* + \left( \widetilde \theta^M_{T,N} - \theta^* \right) \right] = \left[ c + \left( \widetilde c_{T,N} - c \right) \right],
\end{equation}
then the approximated linear system can be seen as a perturbation of the exact one, and we can employ the a priori forward error stability analysis performed in \cite[Section 3.1.2]{QSS00}. Therefore, in the next lemma we quantify the size of the perturbations.
\begin{lemma} \label{lem:bound_perturbation}
Under the same assumptions of \cref{lem:convergence_moments}, there exists a constant $C>0$ independent of $T$ and $N$ such that
\begin{equation}
\begin{aligned}
(i) & \quad \E \left[ \norm{\widetilde{\mathcal A}_{T,N} - \mathcal A} \right] \le C \left( \frac{1}{\sqrt T} + \frac{1}{\sqrt N} \right), \\
(ii) & \quad \E \left[ \norm{\widetilde c_{T,N} - c} \right] \le C \left( \frac{1}{\sqrt T} + \frac{1}{\sqrt N} \right),
\end{aligned}
\end{equation}
where $\norm{\cdot}$ denotes either the spectral norm for a matrix or the Euclidean norm for a vector.
\end{lemma}
\begin{proof}
By the triangle inequality we first have
\begin{equation}
\norm{\widetilde{\mathcal A}_{T,N} - \mathcal A} = \norm{\widetilde{\mathcal M}_{T,N}^\top \widetilde{\mathcal M}_{T,N} - \mathcal M^\top \mathcal M} \le \left( \norm{\widetilde{\mathcal M}_{T,N}}_{\mathrm F} + \norm{\mathcal M}_{\mathrm F} \right) \norm{\widetilde{\mathcal M}_{T,N} - \mathcal M}_{\mathrm F},
\end{equation}
where we also used the fact that $\norm{A} = \norm{A^\top}$ and $\norm{A} \le \norm{A}_{\mathrm F}$ for any matrix $A$, where $\norm{\cdot}_{\mathrm F}$ denotes the Frobenius norm. Then, by Hölder's inequality with exponents $3$ and $3/2$ and by Jensen's inequality we have
\begin{equation} \label{eq:boundAdiff}
\E \left[ \norm{\widetilde{\mathcal A}_{T,N} - \mathcal A} \right] \le \sqrt[3]{4} \left( \E \left[ \norm{\widetilde{\mathcal M}_{T,N}}_{\mathrm F}^3 \right] + \norm{\mathcal M}_{\mathrm F}^3 \right)^{1/3} \left( \E \left[ \norm{\widetilde{\mathcal M}_{T,N} - \mathcal M}_{\mathrm F}^{3/2} \right] \right)^{2/3},
\end{equation}
and it now remains to bound the two expectations in the right-hand side. Due to Jensen's inequality we have
\begin{equation}
\E \left[ \norm{\widetilde{\mathcal M}_{T,N}}_{\mathrm F}^3 \right] = \E \left[ \left( \sum_{i=1}^{M+1} \sum_{j=1}^P \abs{(\widetilde{\mathcal M}_{T,N})_{ij}}^2 \right)^{3/2} \right] \le \sqrt{P(M+1)} \sum_{i=1}^{M+1} \sum_{j=1}^P \E \left[ \abs{(\widetilde{\mathcal M}_{T,N})_{ij}}^3 \right],
\end{equation}
and by using that $(\sum_{i=1}^n x_i)^p \le \sum_{i=1}^n x_i^p$ for $0<p<1$ we get
\begin{equation}
\E \left[ \norm{\widetilde{\mathcal M}_{T,N} - \mathcal M}_{\mathrm F}^{3/2} \right] \le \sum_{i=1}^{M+1} \sum_{j=1}^P \E \left[ \abs{(\widetilde{\mathcal M}_{T,N})_{ij} - (\mathcal M)_{ij}}^{3/2} \right].
\end{equation}
We now notice that the components of the matrices $\widetilde{\mathcal M}_{T,N}$ and $\mathcal M$ consist of sum and products of the moments $\mathbb M^{(r)}$ and their approximations $\widetilde{\mathbb M}^{(r)}_{T,N}$, and therefore due to the boundedness of the moments and by \cref{lem:convergence_moments} we obtain
\begin{equation} \label{eq:boundsMdiffM}
\E \left[ \norm{\widetilde{\mathcal M}_{T,N}}_{\mathrm F}^3 \right] \le C \qquad \text{and} \qquad \E \left[ \norm{\widetilde{\mathcal M}_{T,N} - \mathcal M}_{\mathrm F}^{3/2} \right] \le C \left( \frac{1}{\sqrt T} + \frac{1}{\sqrt N} \right)^{3/2},
\end{equation}
which together with equation \eqref{eq:boundAdiff} implies point $(i)$. For point $(ii)$ we proceed analogously and by triangle inequality we have
\begin{equation}
\norm{\widetilde c_{T,N} - c} = \norm{\widetilde{\mathcal M}_{T,N}^\top \widetilde b_{T,N} - \mathcal M^\top b} \le \norm{\widetilde{\mathcal M}_{T,N}}_{\mathrm F} \norm{\widetilde b_{T,N} - b} + \norm{b} \norm{\widetilde{\mathcal M}_{T,N} - \mathcal M}_{\mathrm F}.
\end{equation}
Then, by Hölder's inequality with exponents $3$ and $3/2$ we have
\begin{equation} \label{eq:boundCdiff}
\begin{aligned}
\E \left[ \norm{\widetilde c_{T,N} - c} \right] &\le \left( \E \left[ \norm{\widetilde{\mathcal M}_{T,N}}_{\mathrm F}^3 \right] \right)^{1/3} \left( \E \left[ \norm{\widetilde b_{T,N} - b}^{3/2} \right] \right)^{2/3} \\
&\quad + \norm{b} \left( \E \left[ \norm{\widetilde{\mathcal M}_{T,N} - \mathcal M}_{\mathrm F}^{3/2} \right] \right)^{2/3},
\end{aligned}
\end{equation}
where the first and last terms are bounded by \eqref{eq:boundsMdiffM}. For the remaining term we have
\begin{equation} \label{eq:boundBdiff}
\E \left[ \norm{\widetilde b_{T,N} - b}^{3/2} \right] \le \sum_{i=1}^{M+1} \E \left[ \abs{(\widetilde b_{T,N})_i - (b)_i}^{3/2} \right] \le C \left( \frac{1}{\sqrt T} + \frac{1}{\sqrt N} \right)^{3/2},
\end{equation}
where we used \cref{lem:convergence_moments} and the fact that the components of the vectors $\widetilde b_{T,N}$ and $b$ consist of sum and products of the moments $\mathbb M^{(r)}$ and their approximations $\widetilde{\mathbb M}^{(r)}_{T,N}$. We remark that the vectors $\widetilde b_{T,N}$ and $b$ depend also on the quantities $Q_{T,N}$ and $Q$ given in \eqref{eq:equationQV} and \eqref{eq:equationQV_limit}, which in turn depend on the approximated and exact moments, respectively, and therefore \cref{lem:convergence_moments} can still be employed. Finally, equation \eqref{eq:boundCdiff} together with \eqref{eq:boundsMdiffM} and \eqref{eq:boundBdiff} yields point $(ii)$ and concludes the proof.
\end{proof}
We are now ready to state and prove the main result of this section, i.e., the asymptotic unbiasedness of our estimator $\widehat \theta_{T,N}^M$ and its rate of of convergence with respect to the final time $T$ and the number of interacting particles $N$ of the system.
\begin{theorem} \label{thm:main}
Let the assumptions of \cref{lem:convergence_moments} hold and let $\det(\mathcal A) \neq 0$. Then, there exists a constant $C>0$ independent of $T$ and $N$ such that
\begin{equation}
\E \left[ \norm{\widehat \theta_{T,N}^M - \theta^*} \right] \le C \left( \frac{1}{\sqrt T} + \frac{1}{\sqrt N} \right),
\end{equation}
and therefore 
\begin{equation}
\lim_{T\to\infty,N\to\infty} \widehat \theta_{T,N}^M = \theta^*, \qquad \text{in } L^1.
\end{equation}
\end{theorem}
\begin{proof}
Let us first define the event $B$ as
\begin{equation}
B = \left\{ \norm{\mathcal A^{-1}} \norm{\widetilde{\mathcal A}_{T,N} - \mathcal A} < \frac12 \right\},
\end{equation}
and notice that by Markov's inequality and \cref{lem:bound_perturbation} we have
\begin{equation} \label{eq:boundPAC}
\Pr(B^\compl) \le 2 \norm{\mathcal A^{-1}} \E \left[ \norm{\widetilde{\mathcal A}_{T,N} - \mathcal A} \right] \le C \norm{\mathcal A^{-1}} \left( \frac{1}{\sqrt T} + \frac{1}{\sqrt N} \right).
\end{equation}
Then, by the law of total expectation we obtain
\begin{equation} 
\E \left[ \norm{\widehat \theta_{T,N}^M - \theta^*} \right] = \E \left[ \left. \norm{\widehat \theta_{T,N}^M - \theta^*} \right| B \right] \Pr(B) + \E \left[ \left. \norm{\widehat \theta_{T,N}^M - \theta^*} \right| B^\compl \right] \Pr(B^\compl),
\end{equation}
which since $\widehat \theta_{T,N}^M, \theta^* \in \Theta$, a compact set, and due to estimate \eqref{eq:boundPAC} implies
\begin{equation} \label{eq:decompositionEtheta}
\E \left[ \norm{\widehat \theta_{T,N}^M - \theta^*} \right] \le \E \left[ \left. \norm{\widehat \theta_{T,N}^M - \theta^*} \right| B \right] + C \norm{\mathcal A^{-1}} \left( \frac{1}{\sqrt T} + \frac{1}{\sqrt N} \right).
\end{equation}
It now remains to study the first term in the right-hand side. Since $\widehat \theta^M_{T,N}$ is the projection of $\widetilde \theta^M_{T,N}$ onto the convex set $\Theta$, we have
\begin{equation}
\norm{\widehat \theta_{T,N}^M - \theta^*} \le \norm{\widetilde \theta_{T,N}^M - \theta^*},
\end{equation}
which yields
\begin{equation}
\E \left[ \left. \norm{\widehat \theta_{T,N}^M - \theta^*} \right| B \right] \le \E \left[ \left. \norm{\widetilde \theta_{T,N}^M - \theta^*} \right| B \right].
\end{equation}
Then, by \cite[Theorem 3.1]{QSS00} we have
\begin{equation}
\begin{aligned}
\E \left[ \left. \norm{\widehat \theta_{T,N}^M - \theta^*} \right| B \right] &\le \E \left[ \left. \frac{\mathcal K(\mathcal A) \norm{\theta^*}}{1 - \norm{\mathcal A^{-1}} \norm{\widetilde{\mathcal A}_{T,N} - \mathcal A}} \left( \frac{\norm{\widetilde c_{T,N} - c}}{\norm{c}} + \frac{\norm{\widetilde{\mathcal A}_{T,N} - \mathcal A}}{\norm{\mathcal A}} \right) \right| B \right] \\
&\le 2 \mathcal K(\mathcal A) \norm{\theta^*} \left( \frac{\E \left[ \left. \norm{\widetilde c_{T,N} - c} \right| B \right]}{\norm{c}} + \frac{\E \left[ \left. \norm{\widetilde{\mathcal A}_{T,N} - \mathcal A} \right| B \right]}{\norm{\mathcal A}} \right),
\end{aligned}
\end{equation}
where $\mathcal K(\mathcal A)$ denotes the condition number of the matrix $\mathcal A$ defined as $\mathcal K(\mathcal A) = \norm{\mathcal A} \norm{\mathcal A^{-1}}$. Then, employing the inequality $\E[Y|B] \le \E[Y]/\Pr(B)$, which holds for any positive random variable $Y$, and estimate \eqref{eq:boundPAC} we obtain
\begin{equation}
\E \left[ \left. \norm{\widehat \theta_{T,N}^M - \theta^*} \right| B \right] \le C \mathcal K(\mathcal A) \left( \E \left[ \norm{\widetilde c_{T,N} - c} \right] + \E \left[ \norm{\widetilde{\mathcal A}_{T,N} - \mathcal A} \right] \right),
\end{equation}
which due to \cref{lem:bound_perturbation} implies 
\begin{equation} \label{eq:boundEtheta}
\E \left[ \left. \norm{\widehat \theta_{T,N}^M - \theta^*} \right| B \right] \le C \mathcal K(\mathcal A) \left( \frac{1}{\sqrt T} + \frac{1}{\sqrt N} \right).
\end{equation}
Finally, the desired results follow from equations \eqref{eq:decompositionEtheta} and \eqref{eq:boundEtheta}.
\end{proof}

\begin{remark} \label{rem:condition}
From the proof of \cref{thm:main} we notice that the precision of our estimator is affected by the condition number of the matrix $\mathcal A$. Moreover, we remark that the inverse of $\widetilde{\mathcal A}_{T,N}$ is never required nor written in the proof of \cref{thm:main}, but it is hidden in the proof of \cite[Theorem 3.1]{QSS00}, where it can be bounded under the condition specified by the set $B$.
\end{remark}

The assumption that the matrix $\mathcal A$ is nonsingular is fundamental for the solvability of the inference problem. In particular, if the matrix $\mathcal A$ is noninvertible, then it means that the parameters cannot be uniquely identified through the method of moments, and thus our approach cannot be applied in this context. Our methodology, indeed, strongly relies on the McKean--Vlasov SDE and therefore it is only possible to estimate the parameters which are also present in the mean field dynamics. A particular example about this fact is given in \cref{rem:OU_identifiable}. However, this assumption is not verifiable for most of the problems since the exact moments with respect to the invariant distributions of the mean field dynamics are usually unknown. In order to check in practice whether the assumption $\det(\mathcal A) \neq 0$ is likely to be satisfied, we suggest to compute the condition number of the matrix $\widetilde A_{T,N}$, which is defined as $\mathcal K(\widetilde A_{T,N}) = \norm{\widetilde A_{T,N}} \norm{\widetilde A_{T,N}^\dagger}$, to verify if the linear system is well-conditioned. If this is not the case, then it is likely that not all the parameters can be identified, at least through the method of moments. Our assumption about the nonsingularity of the matrix $\mathcal A$ seems to be related to the invertibility of the normalized Fisher information matrix in \cite{DeH23}, at least in presence of constant diffusion. Under this condition, the authors provide a detailed analysis about identifiability of the parameters through the maximum likelihood estimator, in particular when the log-likelihood function is quadratic in the parameters \cite[Proposition 16]{DeH23}. However, differently from us, they assume to observe all the particles in a finite time horizon, and not one particle in the long time limit, so they do not need to consider the McKean--Vlasov SDE at stationarity. We also mention the work in \cite{AHP23}, where inference for MckKean--Vlasov SDEs where both the drift and the diffusion coefficients can depend on the law of the process is studied. The authors assume to know discrete observations of the particles in the interacting system and propose a minimal contrast estimator, which is asymptotically unbiased and normal. In order to ensure the identifiability of the parameters, they impose a different condition on some particular quantities which cannot be computed explicitly. Finding more straightforward conditions, which can be explicitly verified a priori, to guarantee the identifiability of the parameters through the method of moments is indeed a difficult, but interesting problem, which we would like to explore in future work.

\section{Numerical experiments} \label{sec:experiments}

In this section we present several numerical experiments to corroborate our theoretical results and show the performance of our estimator is inferring unknown parameters in interacting particle systems with polynomial drift, interaction and diffusion functions. We first perform a sensitivity analysis with respect to the number of moments equations considered, we verify the rate of convergence predicted by \cref{thm:main}, and we compare our methodology with different approaches in the literature. Then, we test our estimator with different examples which may also not fit into the theory, and we finally employ our approach for a system of interacting particles in two dimensions. We generate synthetic data employing the Euler--Maruyama method with a fine time step $h = 0.005$ to compute numerically the solution $\{ (X_t^{(n)})_{t\in[0,T]} \}_{n=1}^N$ of the system of SDEs \eqref{eq:systemIP}. We remark that we always set $X_0^{(n)} = 0$ for all $n = 1, \dots, N$ as initial condition. We then randomly select an index $\bar n \in \{ 1, \dots, N \}$ and we suppose to observe only the sample path $(X_t^{(\bar n)})_{t\in[0,T]}$ of the $\bar n$-th particle of the system, from which we construct the linear system whose solution is the proposed estimator $\widehat \theta_{T,N}^M$.

\subsection{Sensitivity analysis}

\begin{figure}
\centering
\includegraphics[]{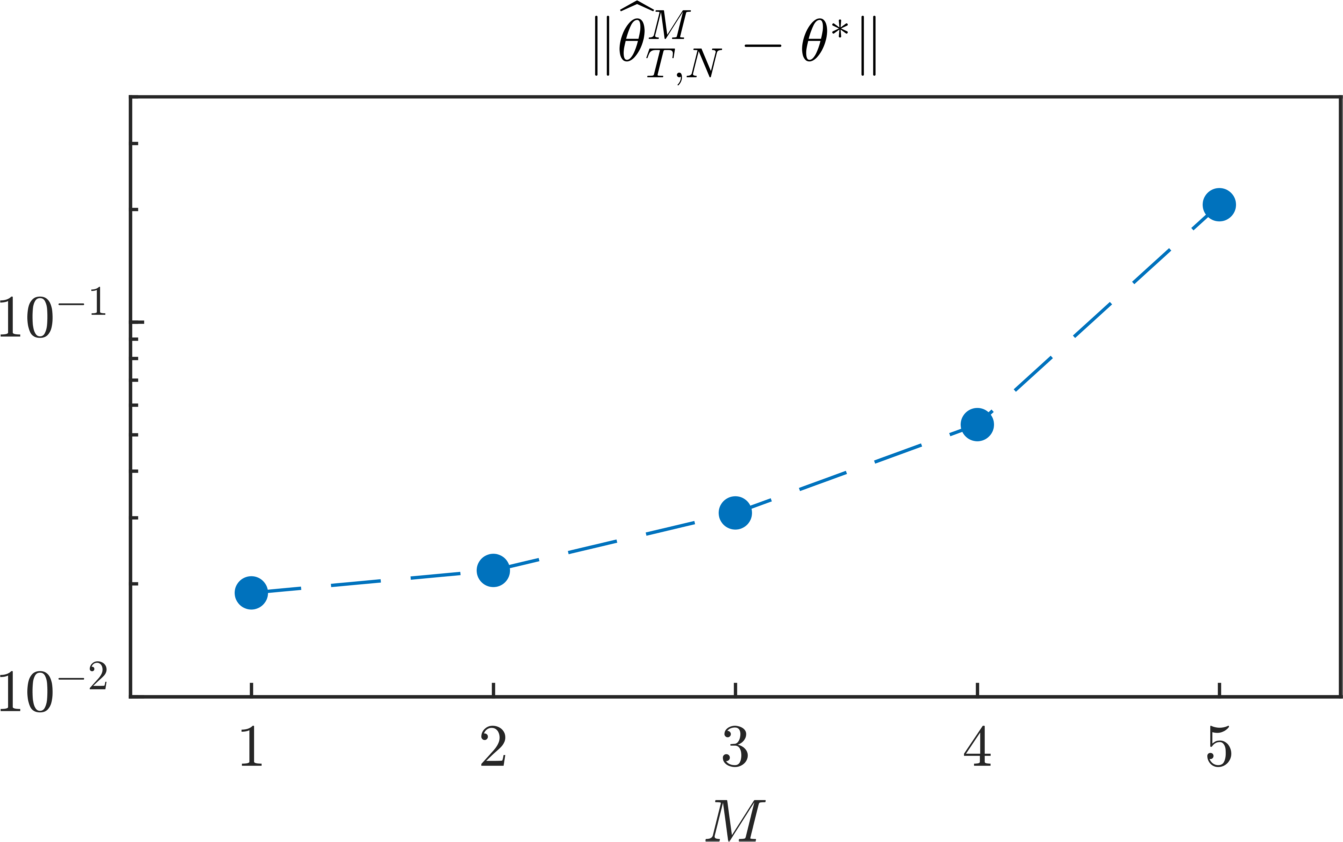} $\,$
\includegraphics[]{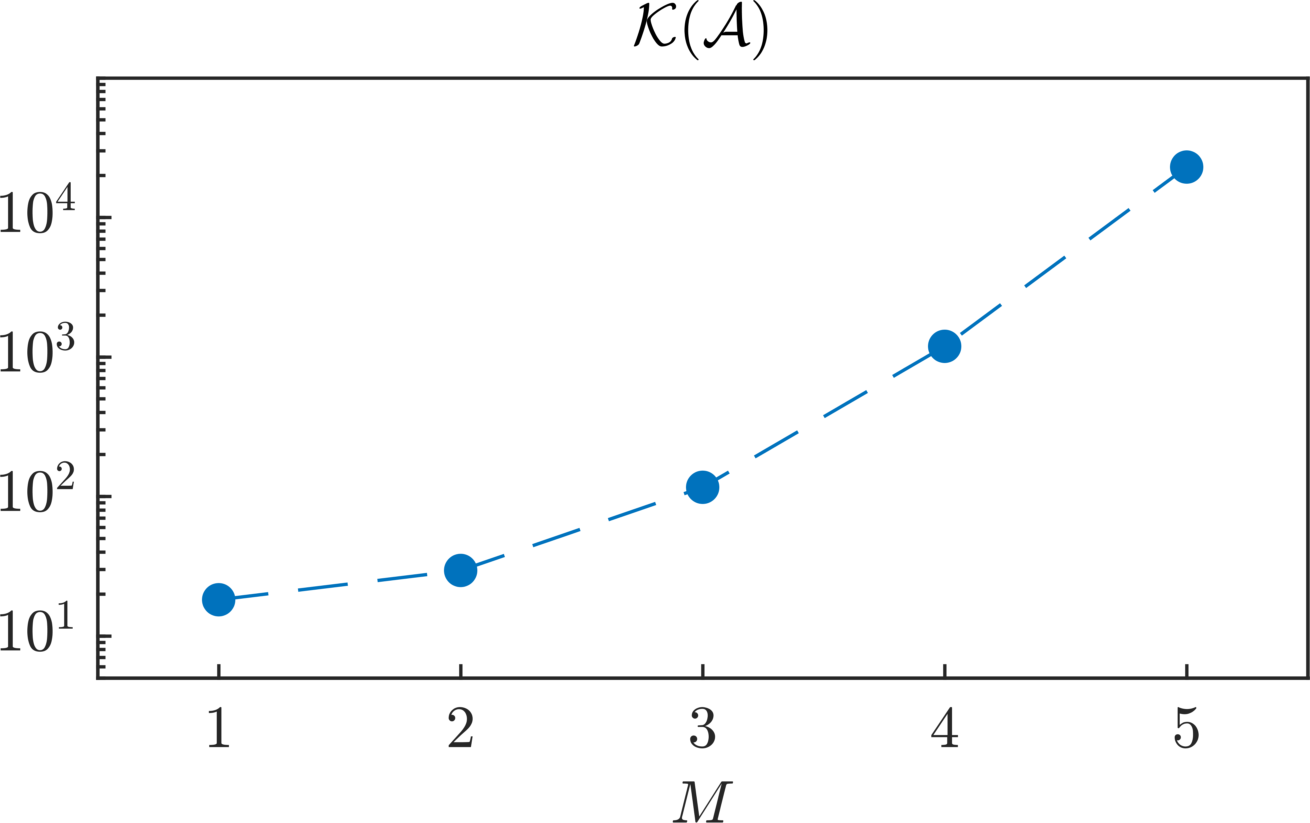}
\caption{Sensitivity analysis for the mean field Ornstein--Uhlenbeck process with respect to the number $M$ of moments equations. Left: error of the estimator $\widehat \theta_{T,N}^M$. Right: condition number of the matrix $\mathcal A$.}
\label{fig:sensitivity_analysis}
\end{figure}

We consider the setting of \cref{sec:OU}, i.e., the mean field Ornstein--Uhlenbeck process, choosing as exact unknown parameters $\alpha_1^* = -1$ and $\sigma_0^* = 1$, and we then aim to estimate $\theta^* = \begin{pmatrix} \alpha_1^* & \sigma_0^* \end{pmatrix}^\top$. Hence, \cref{as:analysis} is satisfied with unique invariant measure of the McKean--Vlasov SDE given by $\mu = \mathcal N(0, 1/2)$, and \cref{as:propagation_chaos} holds true due to \cite[Theorem 3.3]{Mal01} and \cite[Lemma 2.3.1]{Gan08}. In \cref{fig:sensitivity_analysis} we perform a sensitivity analysis with respect to the number $M$ of moments equations employed in the construction of our estimator. We compute $\widehat \theta_{T,N}^M$ for all the particles in the system and for different values of $M = 1,2\dots,5$, fixing the final time $T = 10^4$ and the number of particles $N = 250$. As outlined in \cref{sec:OU}, the linear system which has to be solved is $\widetilde{\mathcal M}_{T,N} \widehat \theta_{T,N}^M = \widetilde b_{T,N}$ where
\begin{equation}
\widetilde b_{T,N} = \begin{pmatrix}
\widetilde{\mathbb M}_{T,N}^{(2)} & \widetilde{\mathbb M}_{T,N}^{(4)} & \cdots & \widetilde{\mathbb M}_{T,N}^{(2M)} & \widetilde Q_{T,N}
\end{pmatrix}^\top,
\end{equation}
and
\begin{equation}
\widetilde{\mathcal M}_{T,N} = \begin{pmatrix}
\widetilde{\mathbb M}_{T,N}^{(2)} & \widetilde{\mathbb M}_{T,N}^{(4)} & \cdots & \widetilde{\mathbb M}_{T,N}^{(2M)} & 0 \\
1 & 3 \widetilde{\mathbb M}_{T,N}^{(2)} & \cdots & (2M-1) \widetilde{\mathbb M}_{T,N}^{(2M-2)} & 1
\end{pmatrix}^\top.
\end{equation}
Moreover, the corresponding exact vector $b$ and matrix $\mathcal M$ for the true values $\alpha_1^*$ and $\sigma_0^*$ are given by
\begin{equation}
b = \begin{pmatrix}
\frac12 & \frac34 & \cdots & \left( \frac12 \right)^M (2M-1)!! & 1
\end{pmatrix}^\top \eqdef \begin{pmatrix}
v^\top & 1
\end{pmatrix}^\top,
\end{equation}
and
\begin{equation}
\mathcal M = \begin{pmatrix}
\frac12 & \frac34 & \cdots & \left( \frac12 \right)^M (2M-1)!! & 0 \\
1 & \frac32 & \cdots & \left( \frac12 \right)^{M-1} (2M-1)!! & 1
\end{pmatrix}^\top \eqdef
\begin{pmatrix}
v^\top & 0 \\
2 v^\top & 1
\end{pmatrix}^\top.
\end{equation}
Therefore, we have
\begin{equation}
\mathcal A = \begin{pmatrix}
\norm{v}^2 & 2 \norm{v}^2 \\
2 \norm{v}^2 & 4 \norm{v}^2 + 1
\end{pmatrix}^\top,
\end{equation}
which implies $\det(\mathcal A) = \norm{v}^2 \neq 0$ as $v \neq 0$, which is the additional assumption in \cref{thm:main}. At the left of \cref{fig:sensitivity_analysis} we then plot the average error and we observe that the inference worsens as the number of moments equations increases. This may sound counterintuitive as one would expect an improvement of the estimation when the number of constraints is higher. However, as highlighted in \cref{rem:condition}, the precision of our estimator is dependent on the condition number of the matrix $\mathcal A$ in the linear system, whose average is plotted at the right of \cref{fig:sensitivity_analysis}. We notice indeed that the condition number increases with $M$ and, in particular, it has the same behavior of the estimation error. Therefore, we believe that the number of constraints should not be much bigger than the minimum number of equations which makes the linear system not underdetermined.

\subsection{Rate of convergence}

\begin{figure}
\centering
\includegraphics[]{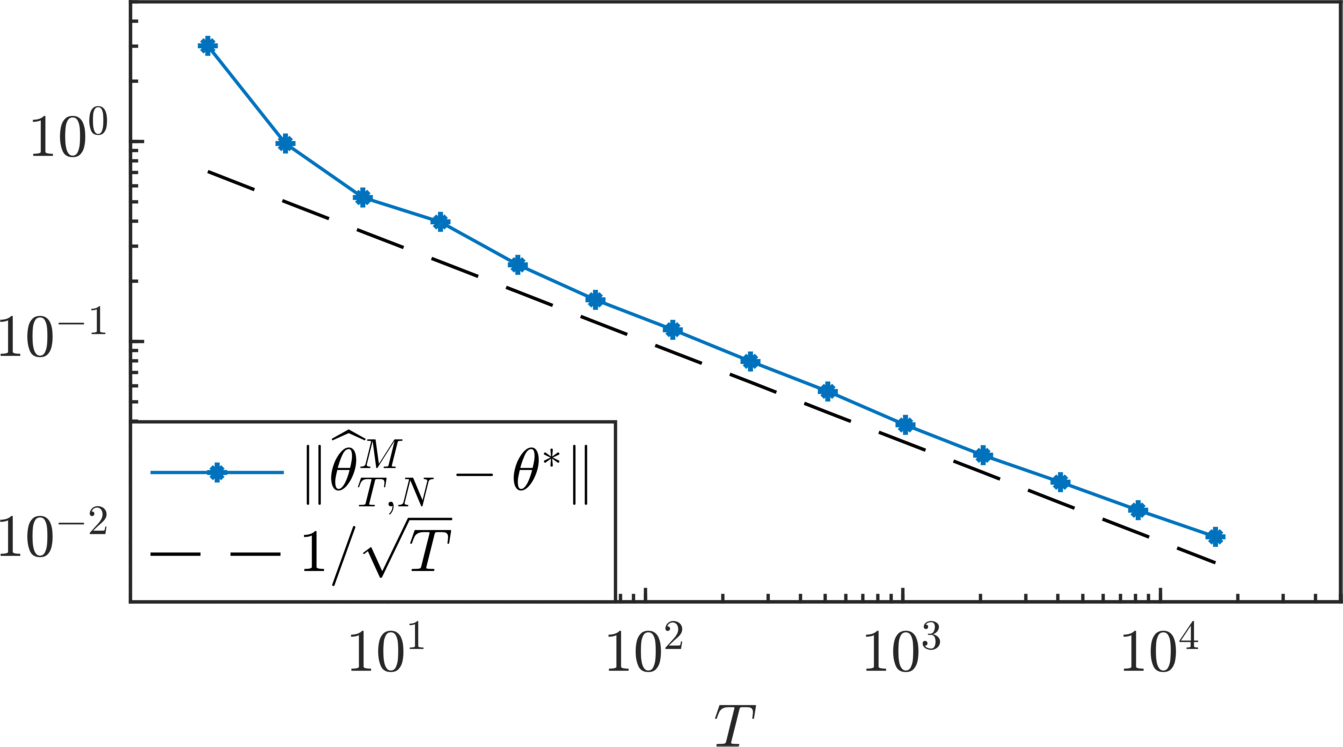} $\,$
\includegraphics[]{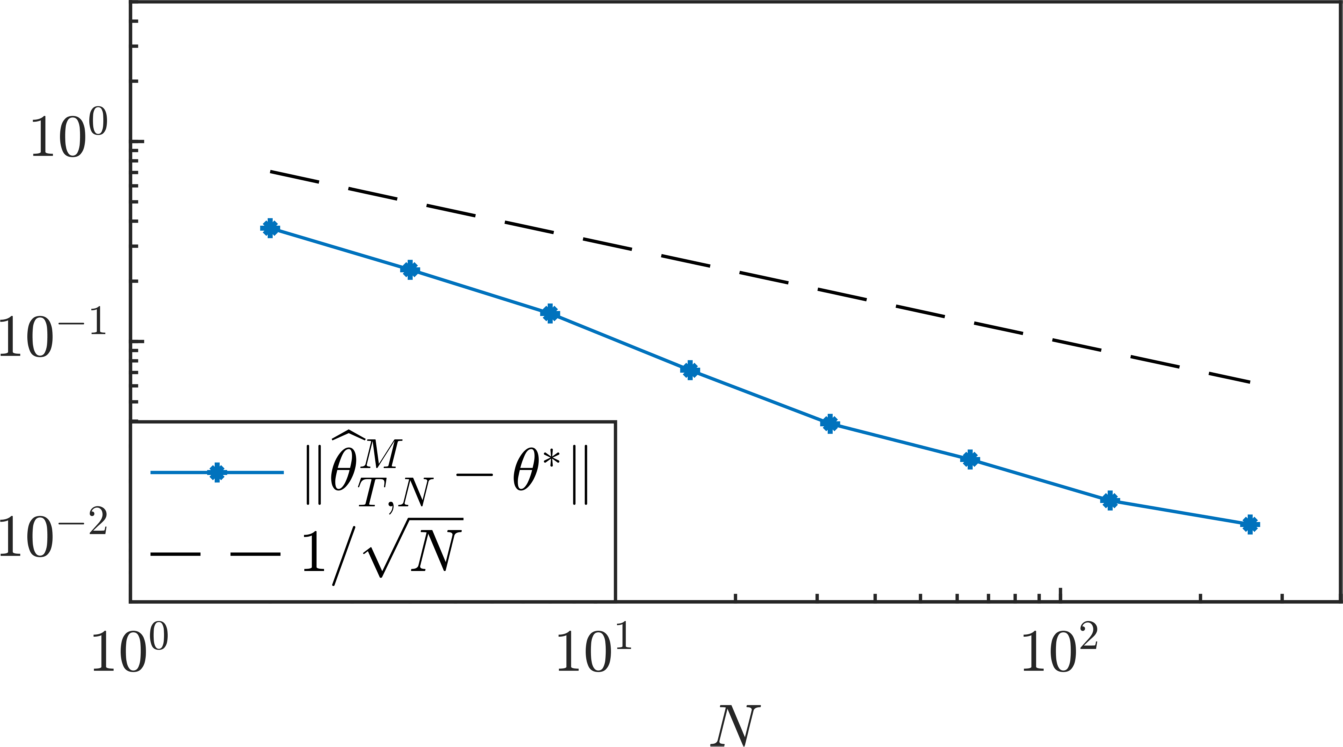}
\caption{Rate of convergence of the estimator $\widehat \theta_{T,N}^M$ towards the exact value $\theta^*$ with respect to the final time $T$ (left) and the number of interacting particles $N$ (right).}
\label{fig:rate_convergence}
\end{figure}

In this section we wish to validate numerically the theoretical result presented in \cref{thm:main}, i.e., the rate of convergence of the estimator with respect to the final time $T$ and the number of particles $N$. We consider the framework of \cref{ex:Malrieu} and we set $V(x;\alpha) = - \alpha_3 x^4 / 4$ and $W(x;\gamma) = x^2/2$, so that $f(x;\alpha) = \alpha_3 x^3$ and $g(x;\gamma) = -x$. \cref{as:analysis,as:propagation_chaos} are guaranteed by \cite{Mal01}. Moreover, we choose $\alpha_3^* = -1$ and $\sigma_0^* = 1$ and we aim to estimate the exact unknown parameter $\theta^* = \begin{pmatrix}
\alpha_3^* & \sigma_0^* \end{pmatrix}^\top$. The numerical results are illustrated in \cref{fig:rate_convergence}, where we use $M = 2$ moments equations with the even moments
\begin{equation}
\begin{aligned}
\widetilde{\mathbb M}_{T,N}^{(4)} \alpha_3 + \sigma_0 &= \widetilde{\mathbb M}_{T,N}^{(2)}, \\
\widetilde{\mathbb M}_{T,N}^{(6)} \alpha_3 + 3 \widetilde{\mathbb M}_{T,N}^{(2)} \sigma_0 &= \widetilde{\mathbb M}_{T,N}^{(4)}, \\
\sigma &= \widetilde Q_{T,N}.
\end{aligned}
\end{equation}
At the left, we fix the number of particles $N = 250$ and we vary the final time $T = 2^i$ with $i = 1,2,\dots,14$. We then plot the average error computed for all the particles in the system as a function of the final time. At the right of the same figure, we fix $T = 10^4$ and we vary $N = 2^i$ with $i = 1,2,\dots,8$. In order to have a fair comparison, we have to compute the average error using the same number of samples, therefore we compute the estimator only for the first particle in the system and we repeat the same procedure $10$ times, plotting then the error as a function of the number of interacting particles. In both cases we observe that the predicted rate of convergence is respected.

\subsection{Comparison with other estimators in case of discrete observations} \label{sec:comparison_OU}

\begin{figure}
\centering
\includegraphics[]{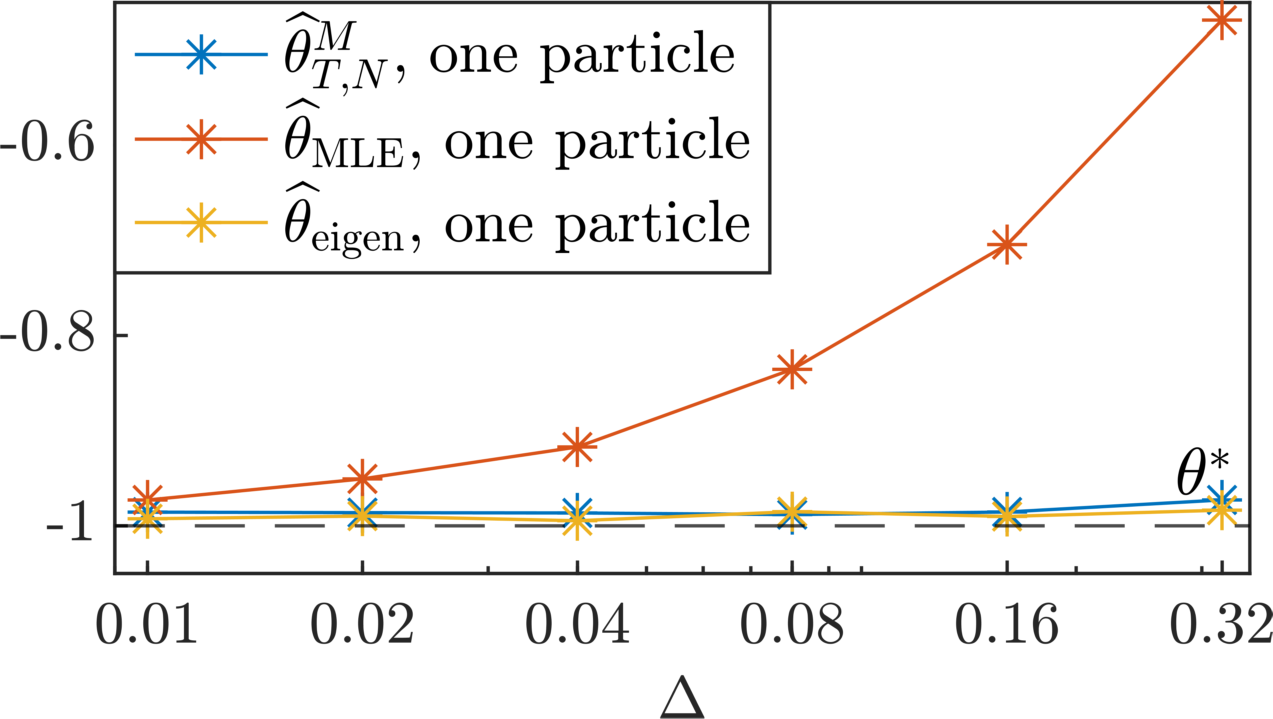} $\,$
\includegraphics[]{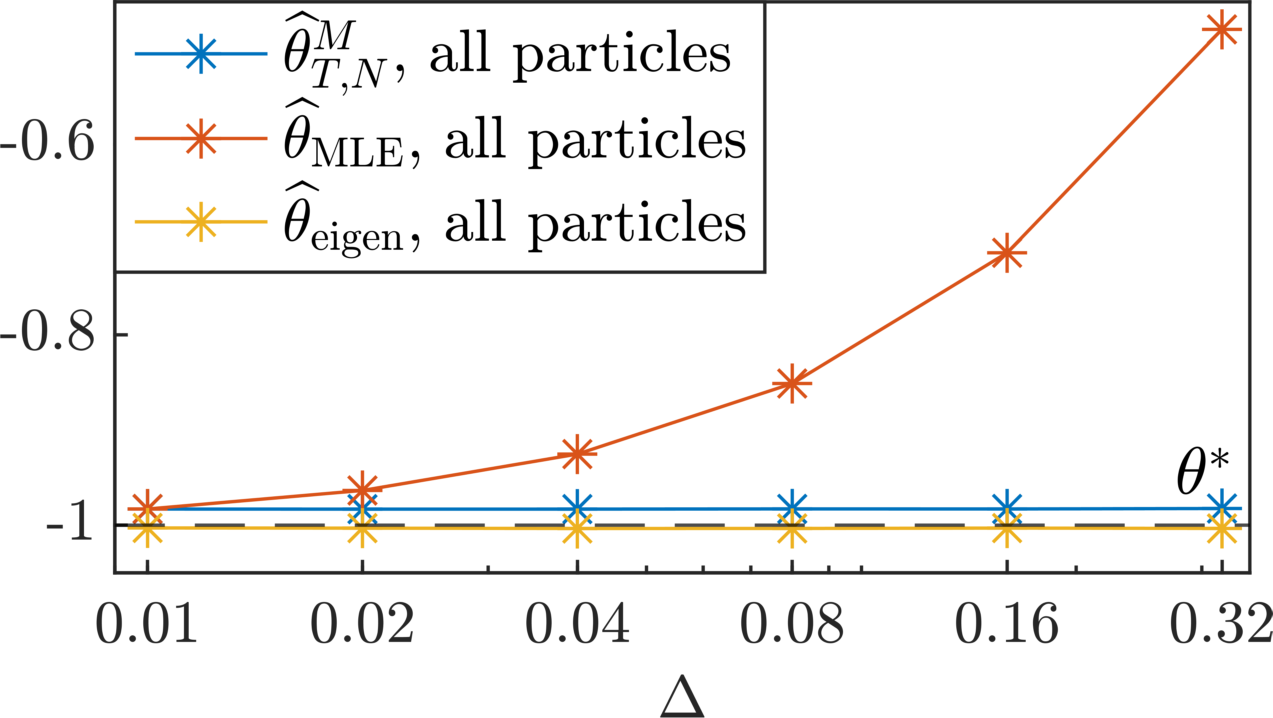}
\caption{Comparison between the estimator $\widehat \theta_{T,N}^M$ with the MLE and the eigenfunction estimator in case of discrete-time observations for different values of the sampling rate $\Delta$. Left: estimation obtained with one particle. Right: average of the estimations obtained with all the particles in the system.}
\label{fig:comparison_OU}
\end{figure}

In this section we compare our methodology with different approaches in the literature. In particular, we consider the maximum likelihood estimator (MLE) in \cite{Kas90} and the eigenfunction estimator in \cite{PaZ22}. We consider the setting of \cref{sec:OU}, we fix the diffusion coefficient $\sigma_0 = 1$ and we aim to efficiently estimate the drift coefficient $\alpha_1$, whose exact value is chosen to be $\alpha_1^* = -1$. Regarding the first estimator, in \cite{Kas90} the author rigorously derives the MLE of the drift coefficient in the context of \cref{sec:OU}, when the path of all the interacting particles in the system is observed. However, since in this work we are assuming to know only the trajectory of one single particle and since for large values of $N$ all the particles are approximately independent and identically distributed, we replace the sample mean with the expectation
with respect to the invariant measure, i.e., $\bar X_t^N$ = 0, and we ignore the sum over all the particles. On the other hand, the second estimator which we consider is suitable for parameter estimation when a sequence of discrete-time observations is given. Therefore, we verify how the three approaches perform for different values of the sampling rate, i.e., the distance between two consecutive observations. We remark that in order to construct our estimator using the method of moments and the MLE, all the integrals have to be discretized. Hence, letting $\Delta$ be the sampling rate and $I = T/\Delta$, we approximate the moments as
\begin{equation}
\widetilde{\mathbb M}_{T,N,\Delta}^{(r)} = \frac{1}{I} \sum_{i=0}^{I-1} (X_{i\Delta}^{(\bar n)})^r.
\end{equation}
Moreover, the other two estimators are written explicitly in \cite[Section 3.2]{PaZ22} and are given by
\begin{equation}
\begin{aligned}
\widehat \theta_{\mathrm{MLE}} &= 1 + \frac{\sum_{i=0}^{I-1} X_{i\Delta}^{(\bar n)} (X_{(i+1)\Delta}^{(\bar n)} - X_{i\Delta}^{(\bar n)})}{\Delta \sum_{i=0}^{I-1} (X_{i\Delta}^{(\bar n)})^2}, \\
\widehat \theta_{\mathrm{eigen}} &= 1 + \frac1\Delta \log \left(\frac{\sum_{i=0}^{I-1} X_{i\Delta}^{(\bar n)} X_{(i + 1) \Delta}^{(\bar n)}}{\sum_{i=0}^{I-1} (X_{i\Delta}^{(\bar n)})^2} \right).
\end{aligned}
\end{equation}
In \cref{fig:comparison_OU} we plot the estimated drift coefficient using only one particle of the system (left) and computing the average of the estimations obtained with all the particles (right) for different values of the sampling rate $\Delta = 0.01 \cdot 2^i$, with $i = 0, \ldots, 5$. We observe that the MLE is biased when the distance between two consecutive observations is not sufficiently small. On the other hand, our estimator and the eigenunction estimator are able to infer the right value of the drift coefficient independently of the sampling rate, and we notice that the eigenfunction estimator seems to give slightly better results. We remark however that, even if in this case the eigenfunction estimator has a closed-form expression, this approach is in general computationally much more expensive than the method of moments, as it requires the solution of the eigenvalue problem for the generator of the mean field limit.

\subsection{Bistable potential} \label{sec:num_bistable}

\begin{figure}
\centering

\hspace{0.5cm}
\begin{overpic}[]{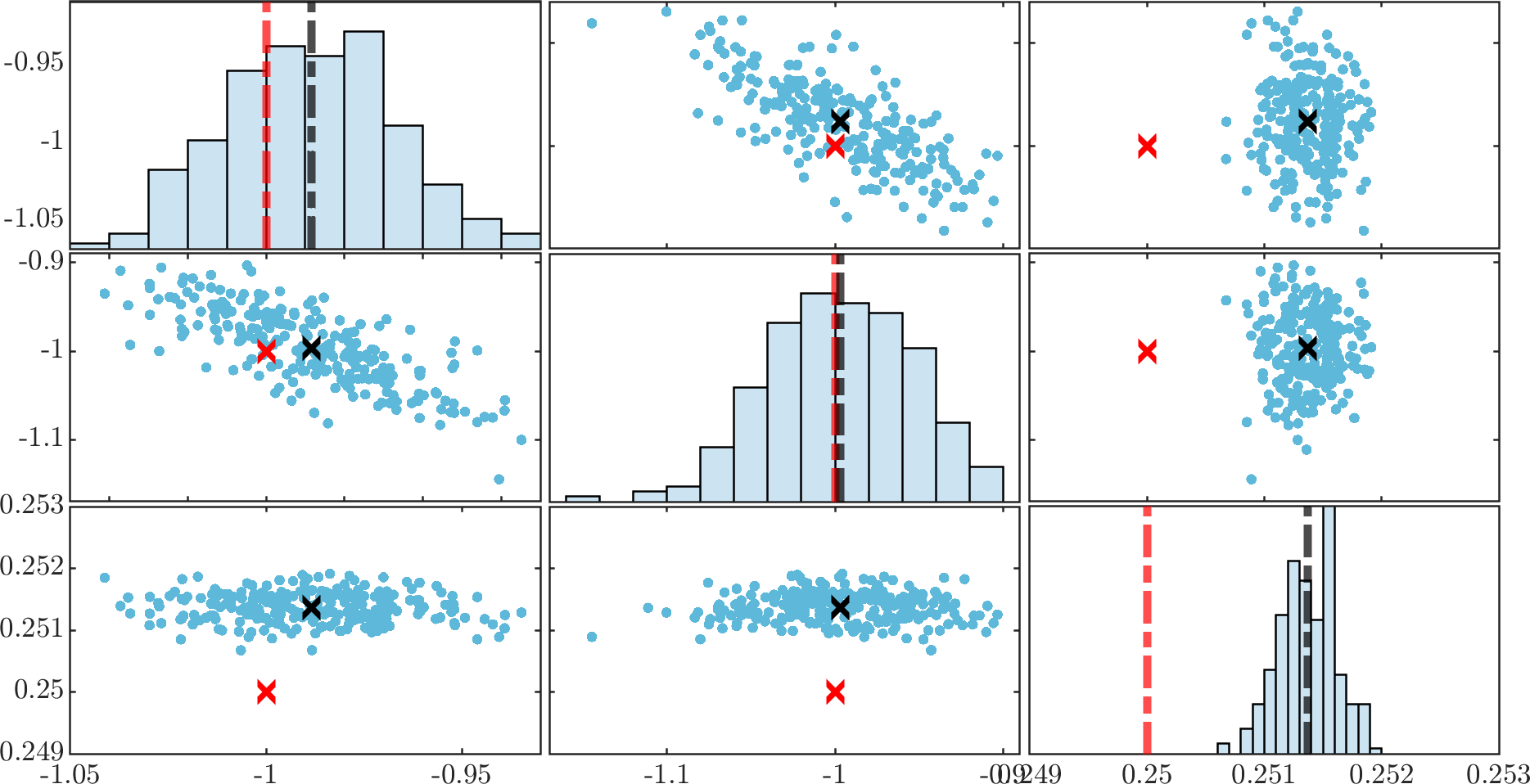} \put(80,-15){$\alpha$} \put(190,-15){$\gamma$} \put(310,-15){$\sigma$} \put(-25,40){$\sigma$} \put(-25,100){$\gamma$} \put(-25,160){$\alpha$} \end{overpic}
\vspace{0.5cm}

\vspace{0.5cm}

\hspace{0.5cm}
\begin{overpic}[]{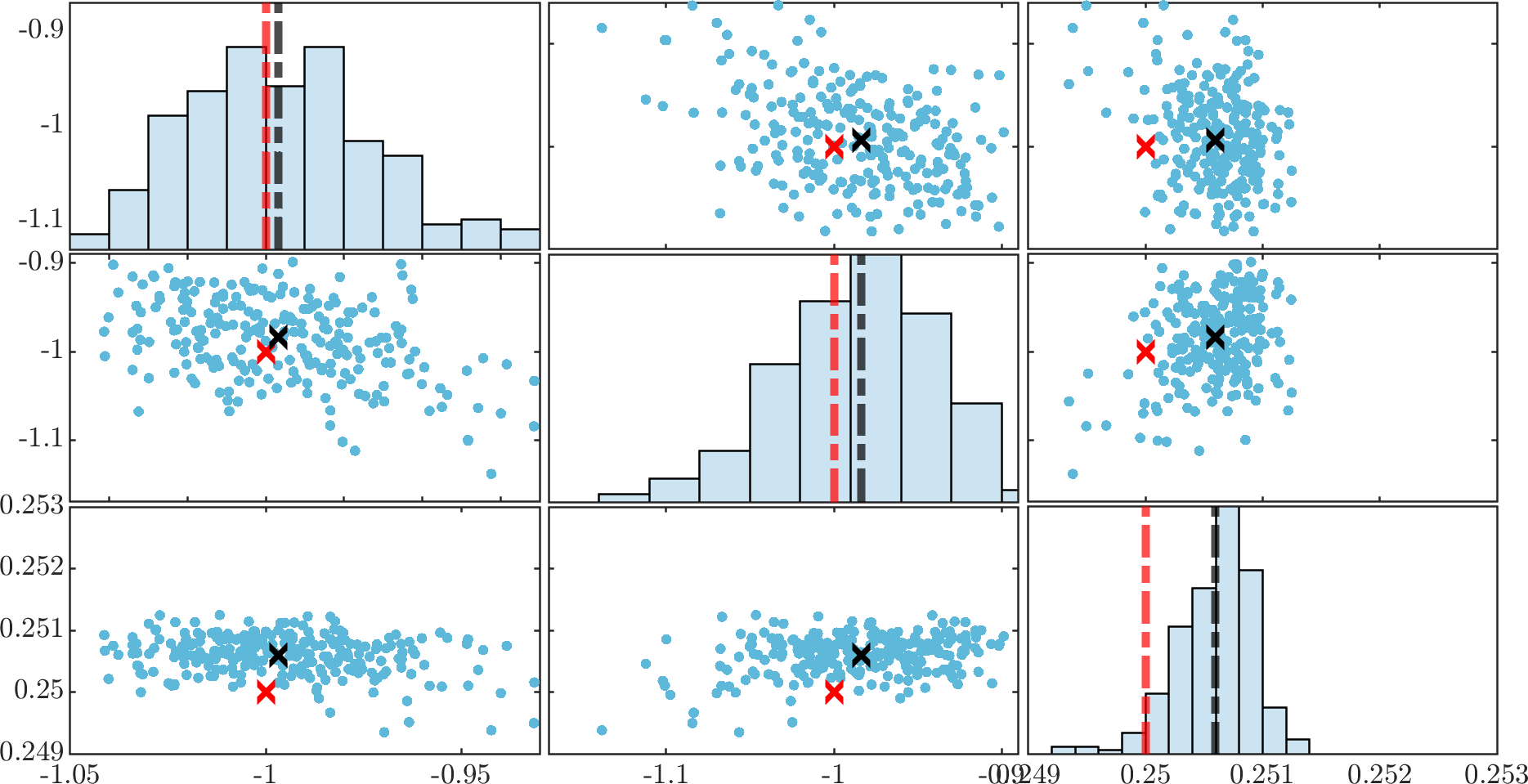} \put(80,-15){$\alpha$} \put(190,-15){$\gamma$} \put(310,-15){$\sigma$} \put(-25,40){$\sigma$} \put(-25,100){$\gamma$} \put(-25,160){$\alpha$} \end{overpic}
\vspace{0.5cm}
\caption{Inference of the drift, interaction and diffusion coefficients for the cases of bistable drift and quadratic interaction (top) and quadratic drift and bistable interaction (bottom) with constant diffusion. Diagonal: histogram of the estimations of each component obtained from all particles. Off-diagonal: scatter plot of the estimations obtained from all particles for two components at a time. Black and red stars/lines represent the average of the estimations and the exact value, respectively.}
\label{fig:bistable}
\end{figure}
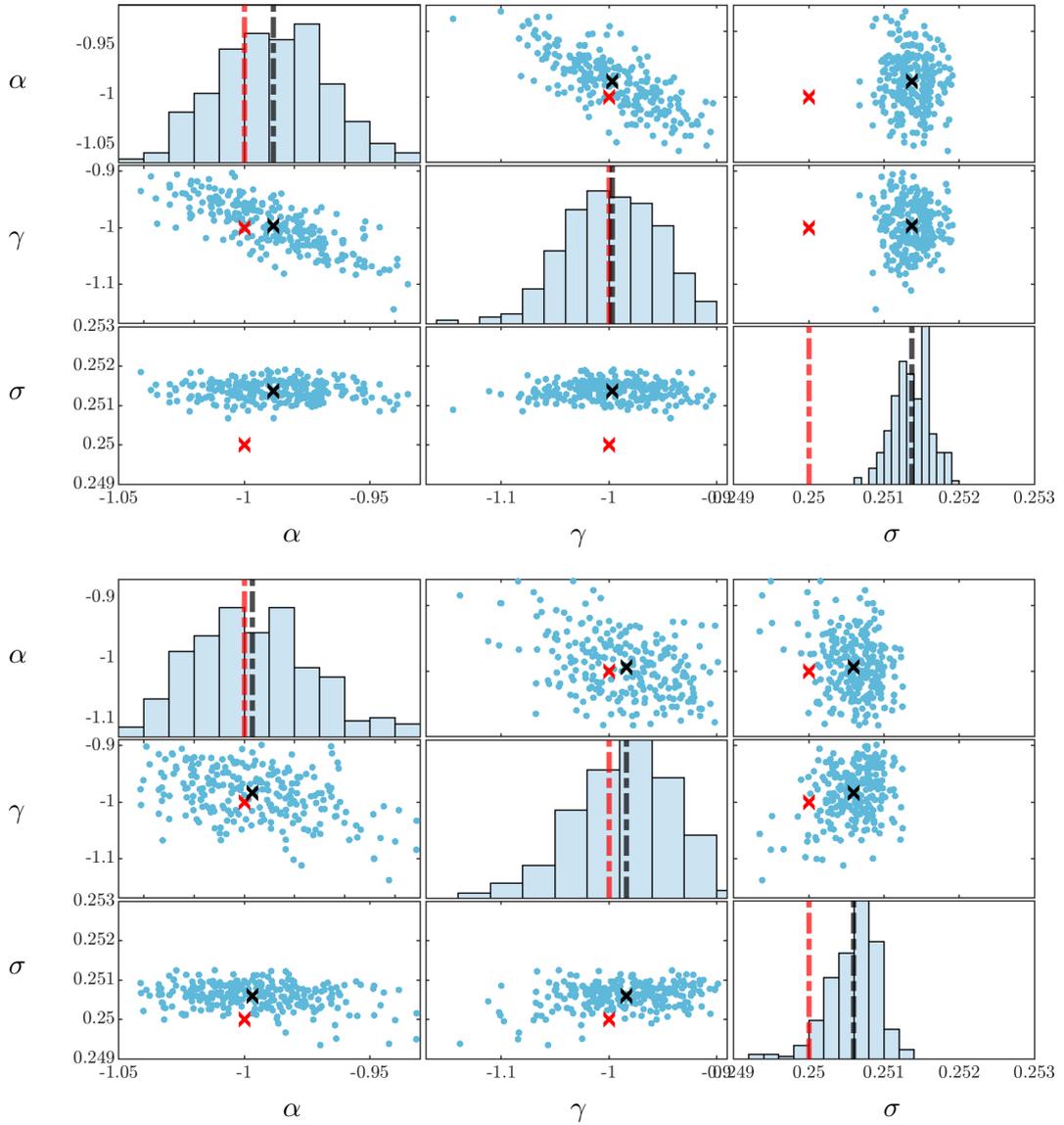

We still consider the setting of \cref{ex:Malrieu}, but we now analyze the bistable potential
\begin{equation}
\mathcal V(x) = \frac{x^4}{4} - \frac{x^2}{2}.
\end{equation}
We consider two cases where the bistable potential appears either in the confining potential or in the interaction potential. In particular, we first set $f_1(x;\alpha) = \alpha \mathcal V'(x)$ and $g_1(x;\gamma) = \gamma x$ and then $f_2(x;\alpha) = \alpha x$ and $g_2(x;\gamma) = \gamma \mathcal V(x)$, and in both cases we set $h(x;\sigma) = \sigma$. Our goal is to estimate the three-dimensional coefficient $\theta^* = \begin{pmatrix} \alpha^* & \gamma^* & \sigma^* \end{pmatrix}^\top$ with exact parameters $\alpha^* = -1$, $\gamma^* = -1$ and $\sigma^* = 0.25$, employing the method of moments with $M = 4$ equations and fixing $T = 10^4$ as final time and $N = 250$ as number of particles. We remark that in this experiments the hypothesis of uniqueness of the invariant measure given in \cref{as:analysis} is not satisfied. We notice that in the first experiment, given this choice of the parameters, the mean field limit admits three different invariant distributions. Nevertheless, the numerical results presented in \cref{fig:bistable}, where we plot the estimations computed for all the particles, suggest that our methodology works even in presence of multiple stationary states. Indeed the exact moments with respect to the right invariant measure are automatically estimated by the empirical moments, and it is not necessary to know a priori the stationary state. In particular, we observe that for the drift and the interaction components the majority of the values are concentrated around the exact unknown, for which their average provides a reasonable approximation. On the other hand, this is not the case for the diffusion coefficient, for which the variance of the estimator is close to zero. However, we notice that the bias is almost negligible and therefore all the particles give a good approximation of the true value. This is caused by the fact that we are including equation \eqref{eq:equationQV} for the quadratic variation, which holds true also for the interacting particles and not only for the mean field limit.

\subsection{Multiplicative noise}

\begin{figure}
\centering
\includegraphics[]{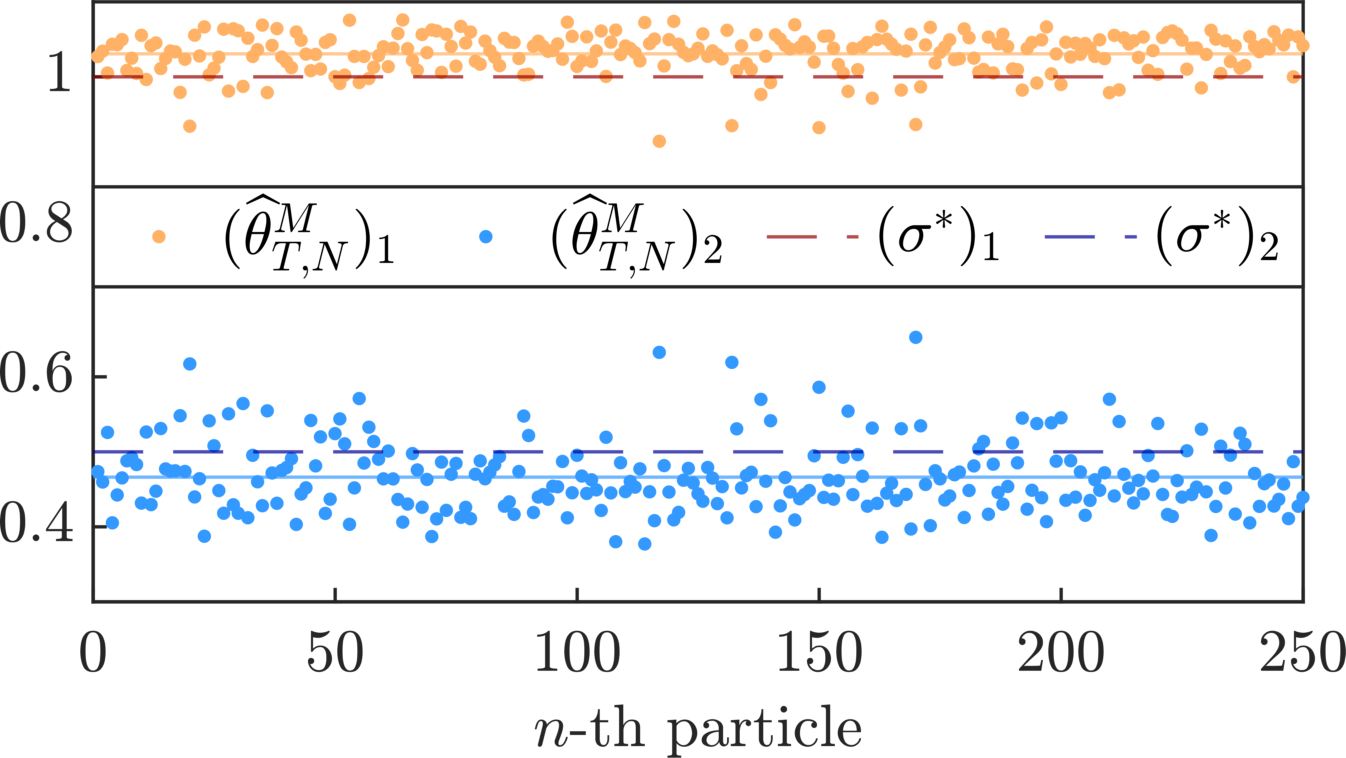}
\caption{Scatter plot for the inference of the two-dimensional diffusion coefficient for the case of simultaneous additive and multiplicative noise.}
\label{fig:multiplicative_noise}
\end{figure}

We now focus on estimating a more complex diffusion term and we consider the functions
\begin{equation}
f(x;\alpha) = -x, \qquad g(x;\gamma) = -x, \qquad \text{and} \qquad h(x;\sigma) = \sigma_0 + \sigma_2 x^2,
\end{equation}
with exact unknown components $\sigma_0^* = 1$ and $\sigma_2^* = 0.5$, and we write $\theta^* = \sigma $. We fix the final time $T = 10^4$ and the number of particles $N = 250$. In \cref{fig:multiplicative_noise} we plot the estimation computed for each particle using $M = 3$ moments equations to construct the estimator
\begin{equation}
\begin{aligned}
\sigma_0 + \widetilde{\mathbb M}_{T,N}^{(2)} \sigma_2 &= 2 \widetilde{\mathbb M}_{T,N}^{(2)} - (\widetilde{\mathbb M}_{T,N}^{(1)})^2, \\
2 \widetilde{\mathbb M}_{T,N}^{(1)} \sigma_0 + 2 \widetilde{\mathbb M}_{T,N}^{(3)} \sigma_2 &= 2 \widetilde{\mathbb M}_{T,N}^{(3)} - \widetilde{\mathbb M}_{T,N}^{(1)} \widetilde{\mathbb M}_{T,N}^{(2)}, \\
3 \widetilde{\mathbb M}_{T,N}^{(2)} \sigma_0 + 3 \widetilde{\mathbb M}_{T,N}^{(3)} \sigma_2 &= 2 \widetilde{\mathbb M}_{T,N}^{(4)} - \widetilde{\mathbb M}_{T,N}^{(1)} \widetilde{\mathbb M}_{T,N}^{(3)}, \\
\sigma_0 + \widetilde{\mathbb M}_{T,N}^{(2)} \sigma_2 &= \widetilde Q_{T,N}.
\end{aligned}
\end{equation}
We observe that the first and the second components are slightly overestimated and underestimated, respectively, but in average the estimators provide a reliable approximation of the correct unknowns.

\subsection{Interacting Fitzhugh--Nagumo Neurons} \label{sec:num_DaiPra}

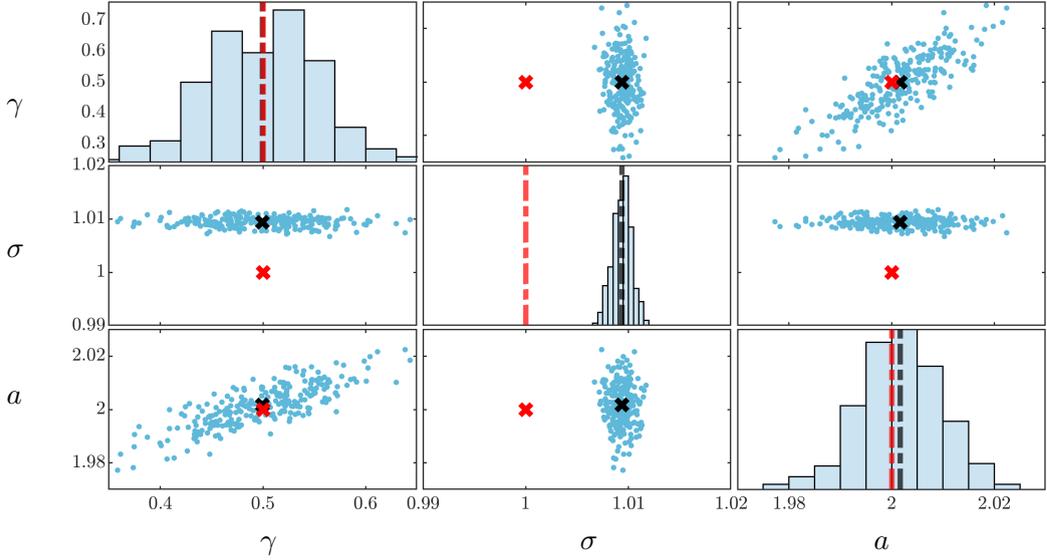
\begin{figure}
\centering
\hspace{0.5cm}
\begin{overpic}[]{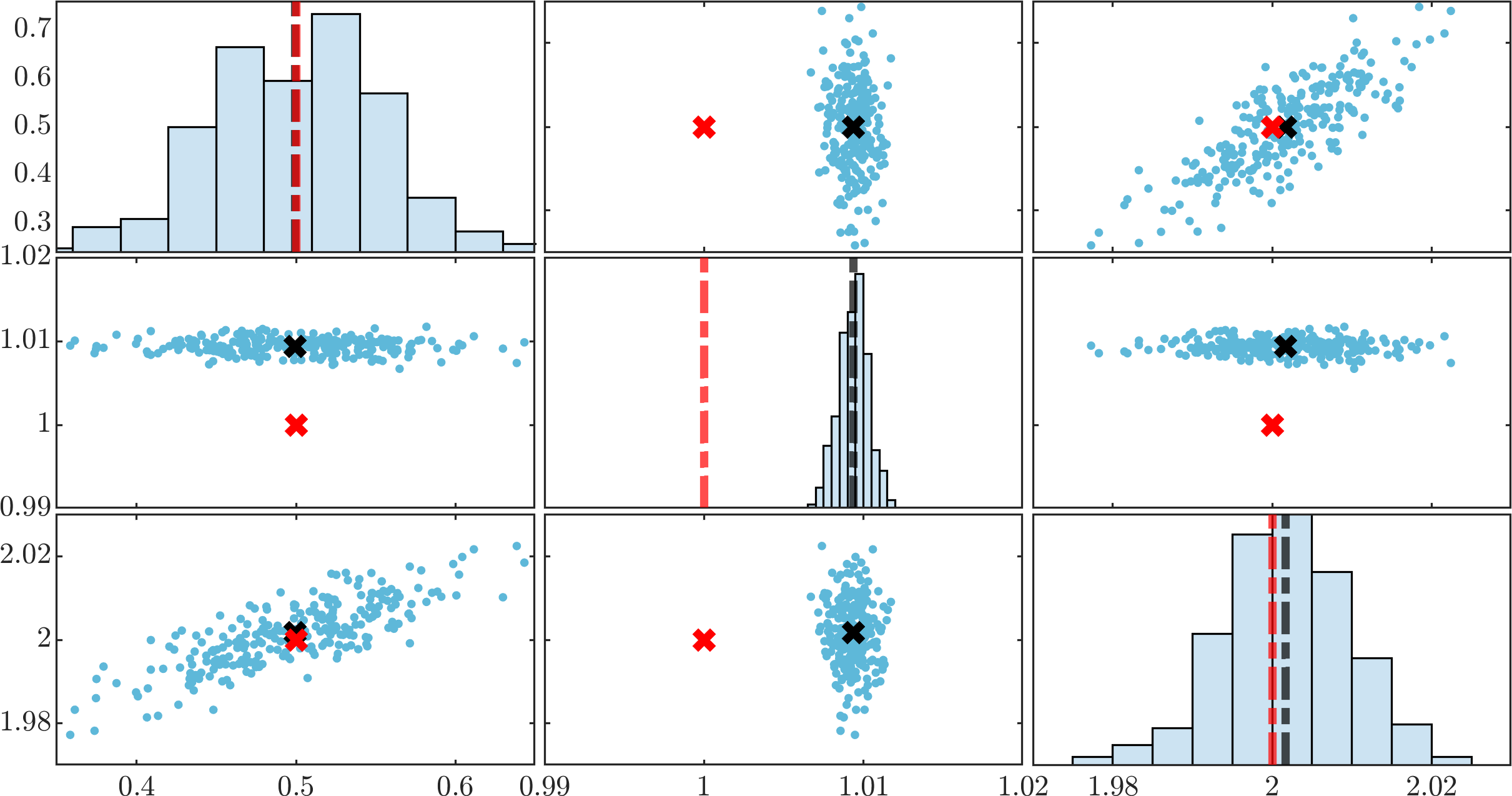} \put(70,-15){$\gamma$} \put(190,-15){$\sigma$} \put(300,-15){$a$} \put(-25,40){$a$} \put(-25,95){$\sigma$} \put(-25,150){$\gamma$} \end{overpic}
\vspace{0.5cm}
\caption{Inference of the interaction, diffusion and kinetic parameters in the model of Fitzhugh--Nagumo for interacting neurons. Diagonal: histogram of the estimations of each component obtained from all particles. Off-diagonal: scatter plot of the estimations obtained from all particles for two components at a time. Black and red stars/lines represent the average of the estimations and the exact value, respectively.}
\label{fig:DaiPra}
\end{figure}

In this last numerical experiment we consider the system of noisy interacting Fitzhugh--Nagumo SDEs which describe the evolution of the membrane potential of interacting neurons \cite[Section 4.2]{Dai19}
\begin{equation}
\begin{aligned}
\d X_t^{(n)} &= \left( X_t^{(n)} - \frac13 (X_t^{(n)})^3 + Y_t^{(n)} \right) \dd t + \gamma \left( X_t^{(n)} - \frac1N \sum_{i=1}^N X_t^{(i)} \right) \dd t + \sqrt{2\sigma} \dd B_t^{(n)}, \\
\d Y_t^{(n)} &= (a - X_t^{(n)}) \dd t,
\end{aligned}
\end{equation}
where $\gamma$ controls the interaction between neurons, $\sigma$ is the diffusion coefficient and $a$ is a kinetic parameter related with input current and synaptic conductance. The corresponding mean field limit then reads
\begin{equation}
\begin{aligned}
\d X_t &= \left( X_t - \frac13 X_t^3 + Y_t \right) \dd t + \gamma \left( X_t - \E[X_t] \right) \dd t + \sqrt{2\sigma} \dd B_t, \\
\d Y_t &= (a - X_t) \dd t.
\end{aligned}
\end{equation}
We note that this mean field SDE is degenerate, since noise acts directly only on the equation for $X_t$. Therefore, the analysis presented in the previous section needs to be modified to take into account the hypoelliptic nature of the dynamics. A rigorous analysis of the mean field Fitzhugh--Nagumo model can be found in \cite{MQT16}. Inference for hypoelliptic mean field SDEs is a very interesting problem that we will return to in future work. The unknown coefficient which we aim to infer is $\theta^* = \begin{pmatrix} \gamma^* & \sigma^* & a^* \end{pmatrix}^\top$ with exact values $\gamma^* = 0.5$, $\sigma^* = 1$ and $a^* = 2$. Moreover, the number of neurons is $N = 250$, the final time is $T = 10^4$ and we employ $M = 4$ moments equations together with the equation given by the quadratic variation
\begin{equation}
\begin{aligned}
a &= \widetilde{\mathbb M}_{T,N}^{(1,0)}, \\
\left( \widetilde{\mathbb M}_{T,N}^{(2,0)} - (\widetilde{\mathbb M}_{T,N}^{(1,0)})^2 \right) \gamma + \sigma &= \frac13 \widetilde{\mathbb M}_{T,N}^{(4,0)} - \widetilde{\mathbb M}_{T,N}^{(2,0)} - \widetilde{\mathbb M}_{T,N}^{(1,1)}, \\
\widetilde{\mathbb M}_{T,N}^{(0,1)} a &= \widetilde{\mathbb M}_{T,N}^{(1,1)}, \\
\left( \widetilde{\mathbb M}_{T,N}^{(1,1)} - \widetilde{\mathbb M}_{T,N}^{(1,0)} \widetilde{\mathbb M}_{T,N}^{(0,1)} \right) \gamma + \widetilde{\mathbb M}_{T,N}^{(1,0)} a &= \frac13 \widetilde{\mathbb M}_{T,N}^{(3,1)} + \widetilde{\mathbb M}_{T,N}^{(2,0)} - \widetilde{\mathbb M}_{T,N}^{(1,1)} - \widetilde{\mathbb M}_{T,N}^{(0,2)}, \\
\sigma &= \widetilde Q_{T,N},
\end{aligned}
\end{equation}
where $\widetilde{\mathbb M}_{T,N}^{(i,j)}$ stands for the empirical approximations of the moments $\E^\mu \left[X^i Y^j\right]$ with respect to the invariant distribution $\mu$ of the mean field dynamics. The numerical results are shown in \cref{fig:DaiPra}, where we plot the estimations obtained with all the interacting neurons. We observe that our methodology is able to correctly infer the unknown parameters, in the sense that the majority of the obtained values are concentrated around the true unknowns. Regarding the diffusion coefficient, we can repeat the same considerations as in \cref{sec:num_bistable}, i.e., that even if the exact parameter is not included in the range of the estimations, the bias is however close to zero.

\section{Conclusion} \label{sec:conclusion}

In this work we considered the framework of large interacting particle systems in one dimension with polynomial confining and interaction potentials, as well as diffusion function, with unknown coefficients. We proposed a novel estimator for inferring these parameters from continuous-time observations of one single particle in the system. Our approach consists in writing a set of linear equations for the unknown coefficients by multiplying the stationary Fokker--Planck equation by monomials, and thus obtaining equations which depend on the moments of the invariant measure of the mean field limit. An approximation of the moments with respect to the invariant state can be obtained by means of the ergodic theorem and using the available path of data, yielding a linear system for the unknown parameters. Moreover, we considered an additional equation employing the definition of the quadratic variation of a stochastic process. We then defined our estimator to be the least squares solution of the final augmented linear system. Under the assumption of ergodicity and uniqueness of the invariant state of the limiting Mckean SDE, we proved that our estimator is asymptotically unbiased when the number of particles and the time horizon tend to infinity, and we also provided a rate of convergence with respect to these two quantities. We remark that this technique is easy to implement since we only need to approximate the moments of the invariant measure of the mean field limit and compute the least squares solution of a linear system. Nevertheless, the numerical experiments presented above demonstrate the accuracy of the obtained estimations even in case of multiple stationary states for the mean field limit and in the multidimensional setting.

This work has a natural interesting direction of research, namely the extension of our methodology to the nonparametric setting, i.e., when the functional form of the drift, interaction and diffusion functions are not known. In particular, if these functions are sufficiently regular, it would be interesting to first approximate them by a truncated Taylor series and then infer the coefficients of the Taylor expansion employing our approach. In this case the theoretical analysis will be based on the study of three simultaneous limits because, in addition to the number of particles and the final time of observation, another quantity of interest will be the dimension of the truncated basis. We will return to this problem in future work.

\subsection*{Acknowledgements} 

We thank the anonymous referees for the extremely careful reading and for their interesting comments and useful suggestions which helped improve and clarify this manuscript. The work of GAP was partially funded by the EPSRC, grant number EP/P031587/1, and by JPMorgan Chase \& Co through a Faculty Award, 2019 and 2021. The work of AZ was partially supported by the Swiss National Science Foundation, under grant No. 200020\_172710.

\bibliographystyle{siamnodash}
\bibliography{biblio}

\end{document}